\pdfoutput=1
\documentclass[reqno]{amsart}

\usepackage{amsmath,amssymb,amsthm,amsfonts,enumerate, enumitem,hyperref,cleveref}
\usepackage{dsfont}
\usepackage[all]{xy}
\usepackage[T1]{fontenc}
\usepackage{tikz}
\usepackage{tikz-cd}
\usepackage{array}
\usepackage{pgfplots}
\pgfplotsset{compat=1.15}
\usepackage{mathrsfs}
\usetikzlibrary{arrows}

\usepackage[margin=1in]{geometry}
\usepackage{caption}

\DeclareMathOperator{\Aut}{Aut}

\DeclareMathOperator{\End}{End}

\DeclareMathOperator{\Span}{Span}

\DeclareMathOperator{\Der}{Der}
\DeclareMathOperator{\IDer}{IDer}
\DeclareMathOperator{\hoch}{HH}

\theoremstyle{plain}
\newtheorem{theorem}{Theorem}[section]
\newtheorem{corollary}[theorem]{Corollary}
\newtheorem{proposition}[theorem]{Proposition}
\newtheorem{lemma}[theorem]{Lemma}

\theoremstyle{definition}
\newtheorem{definition}[theorem]{Definition}
\newtheorem{remark}[theorem]{Remark}
\newtheorem{example}[theorem]{Example}

\crefrangeformat{equation}{#3(#1)#4--#5(#2)#6}

\crefname{theorem}{Theorem}{Theorems}
\crefname{lemma}{Lemma}{Lemmas}
\crefname{corollary}{Corollary}{Corollaries}
\crefname{proposition}{Proposition}{Propositions}
\crefname{definition}{Definition}{Definitions}
\crefname{example}{Example}{Examples}
\crefname{remark}{Remark}{Remarks}
\crefname{conjecture}{Conjecture}{Conjectures}
\crefname{section}{Section}{Sections}
\crefname{equation}{\unskip}{\unskip}
\crefname{enumi}{\unskip}{\unskip}
\crefname{subsection}{Subsection}{Subsections}

\newcommand{\0}{\theta}
\newcommand{\ve}{\varepsilon}

\newcommand{\af}{\alpha}
\newcommand{\bt}{\beta}
\newcommand{\lb}{\lambda}
\newcommand{\Lb}{\Lambda}
\newcommand{\gm}{\gamma}
\newcommand{\vf}{\varphi}

\newcommand{\dl}{\delta}
\newcommand{\Dl}{\Delta}

\newcommand{\sg}{\sigma}

\newcommand{\NN}{\mathds{N}}

\newcommand{\ZZ}{\mathds{Z}}

\newcommand{\m}{{}^{-1}}
\newcommand{\sst}{\subseteq}

\newcommand{\gen}[1]{\langle #1\rangle}

\newcommand{\ol}{\overline}
\newcommand{\ot}{\otimes}
\newcommand{\ch}{\mathrm{char}}
\newcommand{\dett}[1][_q(n)]{{\operatorname{det}#1}}

\newcommand{\id}{\mathrm{id}}
\renewcommand{\iff}{\Leftrightarrow}

\begin{document}
	\title[Quantum upper triangular matrix algebras]{Quantum upper triangular matrix algebras}	
	\author{\'Erica Z. Fornaroli}
	\address{Departamento de Matem\'atica, Universidade Estadual de Maring\'a, Maring\'a, PR, CEP: 87020--900, Brazil}
	\email{ezancanella@uem.br}
	
	\author{Mykola Khrypchenko}
	\address{Departamento de Matem\'atica, Universidade Federal de Santa Catarina,  Campus Trindade, Florian\'opolis, SC, CEP: 88040--900, Brazil}
	\email{nskhripchenko@gmail.com}
	
	\author{Samuel A.\ Lopes}
	\address{CMUP, Departamento de Matem\'atica, Faculdade de Ci\^encias, Universidade do Porto, Rua do Campo Alegre s/n, 4169--007 Porto, Portugal}
	\email{slopes@fc.up.pt}
	
	\author{Ednei A. Santulo Jr.}
	\address{Departamento de Matem\'atica, Universidade Estadual de Maring\'a, Maring\'a, PR, CEP: 87020--900, Brazil}
	\email{easjunior@uem.br}
	
	\subjclass[2020]{16T20, 16S36, 16W20, 16W25, 16E40}
	\keywords{quantum upper triangular matrix algebra, bialgebra, Hopf algebra, automorphism, derivation}
	
	\begin{abstract}
		Following the ideas in~\cite{yM88}, \cite{T90} and inspiration from~\cite{KO24}, we construct a bialgebra $T_q(n)$ and a pointed Hopf algebra $UT_q(n)$ which quantize the coordinate rings of the algebra of upper triangular matrices and of the group of invertible upper triangular matrices of size $n\geq 2$, respectively, where $q$ is a nonzero parameter. The resulting structure on $UT_q(n)$ is neither commutative nor cocommutative and it can be seen as a Hopf quotient of the Takeuchi's two-parameter quantization~\cite{T90} of ${\rm GL}(n)$ corresponding to a specific choice of parameters. The motivation comes from the idea of quantizing the incidence algebra of a finite poset, as the latter can be embedded as a subalgebra of the algebra of upper triangular matrices. We further study and compare the Lie algebras of derivations, the automorphism groups and the low degree Hochschild cohomology of these algebras in case $n=2$.
	\end{abstract}
	
	\maketitle
	
	\tableofcontents
	
	\section*{Introduction}
	
	Incidence algebras and coalgebras, associated to a locally finite poset, are a fundamental object in combinatorics and number theory, as well as in topology, discrete geometry and representation theory (see e.g.~\cite[Chapter 3]{rS12}, \cite{MSS21}). For example, in \cite{C89} and \cite{GS83} the authors show that the Hochschild cohomology of a simplicial complex associated to a finite poset is isomorphic, as a Gerstenhaber algebra, to the Hochschild cohomology of the corresponding incidence algebra (see also \cite{IM22}). More recently, incidence algebras have become a fundamental tool in topological data analysis (see \cite{BL00} for details). Both the algebra and the coalgebra structures of the incidence algebra have been instrumental in combinatorics and number theory, where the poset algebra plays a role analogous to that of the group algebra of a finite group. Since any finite poset has an extension into a totally ordered one, incidence algebras of finite posets are subalgebras of the algebra of upper triangular matrices. 
	
	As a first step towards the introduction of a \textit{quantum incidence algebra}, in this paper we construct a noncommutative and noncocommutative bialgebra $T_q(n)$, axiomatized so that it preserves a uniparameter quantum affine space in such a way that \textit{reflection on the antidiagonal} induces an automorphism. By formally inverting the \textit{quantum determinant}, which in our case is not central, we obtain a Hopf algebra $UT_q(n)$ that coacts on the uniparameter quantum affine space $A_n(q)$, where $q$ is a unit from the base field. It turns out that the quantum group $UT_q(n)$ can be seen as a Hopf quotient of the two-parameter quantization of ${\rm GL}(n)$ introduced by Takeuchi, choosing in~\cite{T90} the parameters $(\af,\bt)=(q^{-1},q)$. It is interesting to note that the choice $(\af,\bt)=(q,q)$ in~\cite{T90} gives the quantum version of ${\rm GL}(n)$ introduced in~\cite{yM88} and~\cite{PW91}, whereas the choice $(\af,\bt)=(1,q)$ gives the version introduced in~\cite{DD91}. To our knowledge, the choice $(\af,\bt)=(q^{-1},q)$ hasn't been studied, nor has its pointed Hopf quotient $UT_q(n)$. This choice of parameters significantly simplifies the relations and yet produces a Hopf algebra which has trivial center in the generic case. We hope this will make it suitable for applications in algebraic combinatorics.
	
	After defining the bialgebra $T_q(n)$ and the Hopf algebra $UT_q(n)$, we explore a few of their properties, including the center and $\ast$-structures. Having \textit{quantum symmetries} in mind, we study the low-dimensional Hochschild cohomology and the automorphism groups of $T_q(n)$ and of $UT_q(n)$ (both as algebras and bialgebras) in case $n=2$.

	\section{Preliminaries}\label{sec-prelim}
	
	Throughout the paper, $K$ will denote a field of characteristic different from $2$ and $K^*$ its multiplicative group of units. All algebras, homomorphisms and tensor products will be considered over $K$, unless otherwise specified. Let $A$ be an algebra. We will use the notation $Z(A)$ for the center of $A$, $\Aut(A)$ for its automorphism group, $\Der(A)$ for its Lie algebra of derivations, $\IDer(A)$ for the ideal of inner derivations and $\hoch^k(A)$ for the $k$-th Hochschild cohomology group of $A$. So, in particular, $\hoch^0(A)=Z(A)$ and $\hoch^1(A)=\Der(A)/\IDer(A)$.
	
	Given a set $X=\{x_i: i\in I\}$, the free unital associative algebra on $X$ will be denoted by $K\gen{X}$ or $K\gen{x_i: i\in I}$. The identity map on the set $X$ is written as $\id_X$ or $\id$, if the set $X$ is clear from the context. We denote by $\ZZ$ the set of integers and by $\NN$ the set of nonnegative integers. 
	
	Recall (cf.~\cite{Alev-Chamarie}) that a (multiparameter) \textit{quantum affine space} of dimension $n\geq 1$ is the quotient of $K\gen{x_1,\dots,x_n}$ by the relations 
	\begin{align}\label{E:AnQ}
		x_ix_j=q_{ij}x_jx_i,\ 1\le i,j\le n,
	\end{align}
	where $Q:=(q_{ij})_{i,j=1}^n\in (K^*)^{n^2}$ is a multiplicatively antisymmetric matrix; in other words, $q_{ji}=q\m_{ij}$ and $q_{ii}=1$, for all $1\le i, j\le n$. We denote this algebra by $A_n(Q)$. 
	
	When $q_{ij}=q$ for all $i>j$, we write $A_n(q)$ instead of $A_n(Q)$.
	Thinking of $(x_i)_{i=1}^n$ in $\left(A_n(q)\right)^n$ as a column vector, with the usual matricial order for the entries, we could represent the relations $x_ix_j=qx_jx_i$, for $i>j$, pictorially as
	\begin{tikzcd}
		j\\ i\arrow["q" ']{u}
	\end{tikzcd}, which has the same meaning as 
	\begin{tikzcd}
		j\arrow["q^{-1}" ]{d}\\ i
	\end{tikzcd}.
	
	\begin{remark}
		Observe that $A_n(Q)$ is isomorphic to the iterated Ore extension $K[x_1][x_2;\sg_2 ]\dots[x_n;\sg_n]$, where $\sg_i$ is the automorphism of $K[x_1]\dots[x_{i-1};\sg_{i-1}]$ given by $\sg_i(x_j)=q_{ij}x_j$ for all $1\le j<i\le n$. In particular, $A_n(Q)$ is a noetherian domain by~\cite[Corollary 2.7]{Goodearl-Warfield} with $K$-basis formed by the (equivalence classes of) monomials $x^\nu:=x_1^{\nu_1}\cdots x_n^{\nu_n}$, where $\nu\in\NN^n$. Moreover, $A_n(Q)$ is a $\ZZ^n$-graded algebra whose homogeneous components are $A_n(Q)_\nu=Kx^\nu$, whenever $\nu\in\NN^n$, and $A_n(Q)_\nu=\{0\}$, otherwise.
	\end{remark}
	
	\section{Quantum $T_q(n)$}
	Let $n\ge 2$ be an integer. We are going to propose a quantization of the algebra of upper triangular $n\times n$ matrices following the ideas in~\cite{yM88} (see also~\cite[Chapter IV]{Kassel}). It is worth noting that this is neither (at least in any obvious way) a subalgebra nor a quotient of the algebra of quantum matrices (cf.\ \cite{RTF89, AST91, PW91}), often denoted by $M_q(n)$ or $\mathcal O_q(M_n(K))$.
	
	Let $a_{ij}$, $1\le i\le j\le n$, be variables that one can organize in the following upper triangular $n\times n$ matrix
	\begin{align*}
		A=
		\begin{pmatrix}
			a_{11} & a_{12} & \dots & a_{1n}\\
			0 & a_{22} & \dots & a_{2n}\\
			\vdots & \ddots & \ddots & \vdots\\
			0 & \dots & 0 & a_{nn}
		\end{pmatrix}.
	\end{align*}
	Fix the map $\rho$ sending $a_{ij}$ to $a_{n+1-j,n+1-i}$. In the language of matrices this is the reflection across the antidiagonal\footnote{It comes from the corresponding involution of the algebra of upper triangular $n\times n$ matrices.},
	\begin{align*}
		\rho(A)
		= 
		\begin{pmatrix}
			a_{nn} & a_{n-1,n} & \dots & a_{1n}\\
			0 & a_{n-1,n-1} & \dots & a_{1,n-1}\\
			\vdots & \ddots & \ddots & \vdots\\
			0 & \dots & 0 & a_{11}
		\end{pmatrix}.
	\end{align*}
	Fix $q\in K^*$ and consider the algebra $K\gen{a_{ij}:1\le i\le j\le n}\otimes_K A_n(q)$.
	Following~\cite{yM88, AST91, Kassel}, we define 
	\begin{align}\label{X'-and-X''}
		x'_i=\sum_{j=i}^na_{ij}\otimes x_j
		\text{ and }
		x''_i=\sum_{j=i}^n a_{n+1-j,n+1-i}\otimes x_j.
	\end{align}
	
	\begin{lemma}\label{relations-in-T_q(n)}
		One has
		$x'_jx'_i=qx'_ix'_j$ and $x''_jx''_i=qx''_ix''_j$, $1\le i<j\le n$,
		if and only if
		\begin{align}
			a_{jk}a_{ik}&=qa_{ik}a_{jk}, \ i<j\le k,\label{a_jka_ik-is-qa_ika_jk}\\
			a_{jk}a_{jl}&=qa_{jl}a_{jk}, \ j\le k<l,\label{a_jka_jl-is-qa_jla_jk}\\
			a_{ik}a_{jl}&=a_{jl}a_{ik}, \ i<j\le l,\ i\le k<l,\label{a_ika_jl-is-a_jla_ik}\\
			a_{jk}a_{il}&=q^2a_{il}a_{jk}, \ i<j\le k<l.\label{a_jka_il-is-a^2a_ila_jk}
		\end{align}  
	\end{lemma}
	\begin{proof}

		Let $i<j$. We have $x'_i=\sum_{k=i}^n a_{ik}\ot x_k$ and $x'_j=\sum_{l=j}^n a_{jl}\ot x_l$, so
		\begin{align*}
			x'_ix'_j&=\sum_{k=i}^n\sum_{l=j}^na_{ik}a_{jl}\ot x_kx_l\\
			&=\sum_{k=j}^na_{ik}a_{jk}\ot x_k^2+\sum_{i\le k<l,\ j\le l}a_{ik}a_{jl}\ot x_kx_l+\sum_{j\le l<k}qa_{ik}a_{jl}\ot x_lx_k\\
			&=\sum_{k=j}^na_{ik}a_{jk}\ot x_k^2+\sum_{i\le k<l,\ j\le l}a_{ik}a_{jl}\ot x_kx_l+\sum_{j\le k<l}qa_{il}a_{jk}\ot x_kx_l\\
			&=\sum_{k=j}^na_{ik}a_{jk}\ot x_k^2+\sum_{i\le k<j\le l}a_{ik}a_{jl}\ot x_kx_l+\sum_{j\le k<l}(a_{ik}a_{jl}+qa_{il}a_{jk})\ot x_kx_l.
		\end{align*}
		Similarly,
		\begin{align*}
			x'_jx'_i
			&=\sum_{k=j}^na_{jk}a_{ik}\ot x_k^2+\sum_{i\le k<j\le l}qa_{jl}a_{ik}\ot x_kx_l+\sum_{j\le k<l}(qa_{jl}a_{ik}+a_{jk}a_{il})\ot x_kx_l.
		\end{align*}
		Thus, $x'_jx'_i=qx'_ix'_j$ if and only if
		\begin{align}
			a_{jk}a_{ik}&=qa_{ik}a_{jk},\ i<j\le k,\label{a_jka_ik=qa_ika_jk}\\
			a_{jl}a_{ik}&=a_{ik}a_{jl},\ i\le k<j\le l,\label{a_jla_ik=a_ika_jl}\\
			a_{ik}a_{jl}-a_{jl}a_{ik}&=q\m a_{jk}a_{il}-qa_{il}a_{jk},\ i<j\le k<l.\label{a_ika_jl-a_jla_ik=q-inv.a_jka_il-qa_ila_jk}
		\end{align}
		The equivalent conditions for $x''_jx''_i=qx''_ix''_j$ are obtained from \cref{a_jka_ik=qa_ika_jk,a_jla_ik=a_ika_jl,a_ika_jl-a_jla_ik=q-inv.a_jka_il-qa_ila_jk} by applying $\rho$ to each variable:
		\begin{align*}
			a_{k'j'}a_{k'i'}&=qa_{k'i'}a_{k'j'},\ k'\le j'<i',\\
			a_{l'j'}a_{k'i'}&=a_{k'i'}a_{l'j'},\ l'\le j'<k'\le i',\\
			a_{k'i'}a_{l'j'}-a_{l'j'}a_{k'i'}&=q\m a_{k'j'}a_{l'i'}-qa_{l'i'}a_{k'j'},\ l'<k'\le j'<i',
		\end{align*}
		where $i'=n+1-i$, $j'=n+1-j$, $k'=n+1-k$ and $l'=n+1-l$. 
		Renaming the indices $(l',k',j',i')$ by $(i,j,k,l)$, we obtain
		\begin{align}
			a_{jk}a_{jl}&=qa_{jl}a_{jk},\ j\le k<l,\label{a_jka_jl=qa_jla_jk}\\
			a_{ik}a_{jl}&=a_{jl}a_{ik},\ i\le k<j\le l,\label{a_ika_jl=a_jla_ik}\\
			a_{jl}a_{ik}-a_{ik}a_{jl}&=q\m a_{jk}a_{il}-qa_{il}a_{jk},\ i<j\le k<l.\label{a_jla_ik-a_ika_jl=q-inv.a_jka_il-qa_ila_jk}
		\end{align}
		In particular, we see that \cref{a_jla_ik=a_ika_jl} is the same as \cref{a_ika_jl=a_jla_ik}, and, due to $\ch(K)\ne 2$, \cref{a_ika_jl-a_jla_ik=q-inv.a_jka_il-qa_ila_jk,a_jla_ik-a_ika_jl=q-inv.a_jka_il-qa_ila_jk} together are equivalent to
		\begin{align*}
			a_{ik}a_{jl}=a_{jl}a_{ik},\ a_{jk}a_{il}=q^2a_{il}a_{jk},\ i<j\le k<l.
		\end{align*}
	\end{proof}
	
	\begin{definition}
		For $n\geq 2$, define $T_q(n)$ to be the quotient of $K\gen{a_{ij}:1\le i\le j\le n}$ by the ideal generated by relations \cref{a_jka_ik-is-qa_ika_jk,a_jka_jl-is-qa_jla_jk,a_ika_jl-is-a_jla_ik,a_jka_il-is-a^2a_ila_jk}.
	\end{definition}
	
	As above for the quantum affine space, we can represent relations \cref{a_jka_ik-is-qa_ika_jk,a_jka_jl-is-qa_jla_jk,a_ika_jl-is-a_jla_ik,a_jka_il-is-a^2a_ila_jk} pictorially as follows:
	\begin{equation}\label{E:rels-digram}
		\begin{tikzcd}[baseline=0pt]
			(i,k) \arrow{r} & (i, l)\\ (j,k) \arrow[Rightarrow]{ur} \arrow{u} \arrow{r} & (j, l)  \arrow{u}
		\end{tikzcd}
	\end{equation}
	The convention used in \cref{E:rels-digram} is that the single horizontal and vertical arrows have label $q$, whereas the diagonal double arrow has label $q^2$; the absence of an edge means that the corresponding generators commute. Keeping the analogy with matrices, it is assumed that $i<j$ and $k<l$.

	\begin{proposition}\label{P:auto-rho}
		The map $\rho$ sending $a_{ij}$ to $a_{n+1-j,n+1-i}$ defines an automorphism of order $2$ of $T_q(n)$.
	\end{proposition}
	\begin{proof}
		Although it could be checked directly that $\rho$ preserves the defining relations of $T_q(n)$, this follows immediately from the symmetry of the diagram \cref{E:rels-digram} under reflection across the antidiagonal. Clearly, $\rho^2=\id_{T_q(n)}$ so $\rho$ is an automorphism of order $2$.
	\end{proof}

	\begin{remark} 
		When $n$ is even, there exists an involution on the algebra of upper triangular matrices which is not equivalent to the reflection across the antidiagonal. This involution is given by the map $\0(a_{ij})=-a_{n+1-j,n+1-i}$, if $i\leq\frac n2<j$, and $\0(a_{ij})=a_{n+1-j,n+1-i}$, otherwise. Applying $\0$ to every variable in equalities \cref{a_jka_ik=qa_ika_jk,a_jla_ik=a_ika_jl,a_ika_jl-a_jla_ik=q-inv.a_jka_il-qa_ila_jk}, one obtains the same conditions given by equalities \cref{a_jka_jl=qa_jla_jk,a_ika_jl=a_jla_ik,a_jla_ik-a_ika_jl=q-inv.a_jka_il-qa_ila_jk}. 
	\end{remark}
	
	\begin{corollary}
		The algebra $T_q(n)$ is isomorphic to $A_{\binom{n+1}2}(Q)$ for some $Q=Q(q)$, where $\{a_{ij}:1\le i\le j\le n\}$ is ordered lexicographically. 
	\end{corollary}
	
	\begin{example}\label{ex-T_q(2)}
		The algebra $T_q(2)$ is the quotient of $K\gen{a_{11},a_{12},a_{22}}$ by the ideal generated by
		\begin{align}
			a_{22}a_{12}&=qa_{12}a_{22},\label{a_22a_12-is-qa_12a_22}\\
			a_{11}a_{12}&=qa_{12}a_{11},\label{a_11a_12-is-qa_12a_11}\\
			a_{11}a_{22}&=a_{22}a_{11}.\label{a_11a_22-is-a_22a_11}
		\end{align}
		Observe that there are no relations of the form \cref{a_jka_il-is-a^2a_ila_jk}. It follows that $T_q(2)$ is isomorphic to $A_{3}(Q)$, where $Q=\left(\begin{smallmatrix} 1&q&1\\q^{-1}&1&q^{-1}\\ 1&q&1\end{smallmatrix}\right)$.
		The relations \cref{a_22a_12-is-qa_12a_22,a_11a_12-is-qa_12a_11,a_11a_22-is-a_22a_11} can be represented by the diagram\quad  \begin{tikzcd}
			a_{11} \arrow{r} & a_{12}\\  & a_{22}  \arrow{u}
		\end{tikzcd}.
	\end{example}

	Unlike the case of the quantum full matrix algebra~\cite[Proposition IV.3.3]{Kassel}, the center of $T_q(n)$ is trivial whenever $q$ is not a root of unity.
	
	\begin{proposition}\label{center-T_q(n)}
		If $q$ is not a root of unity, then $Z(T_q(n))=K$.
	\end{proposition}
	\begin{proof}
		By \cite[p.~1788]{Alev-Chamarie}, $Z(T_q(n))$ is the $K$-space generated by those monomials 
		\begin{equation*}
			f=a_{11}^{\nu_{11}}\cdots a_{1n}^{\nu_{1n}}a_{22}^{\nu_{22}}\cdots a_{2n}^{\nu_{2n}}\cdots a_{n-1,n-1}^{\nu_{n-1,n-1}}a_{n-1,n}^{\nu_{n-1,n}}a_{nn}^{\nu_{nn}}, 
		\end{equation*}
		with $\nu_{kl}\in \NN$, which are central. 
		
		Now, $a_{11}f=q^{\sum_{i=2}^n\nu_{1i}} fa_{11}$, so $\sum_{i=2}^n\nu_{1i}=0$ and $\nu_{1i}=0$ for all $2\leq i\leq n$. Similarly, using the above, we have $a_{12}f=q^{-\nu_{11}-\nu_{22}}fa_{12}$, so $\nu_{11}=0=\nu_{22}$. This established the result in case $n=2$. Suppose that $n>2$. We have shown that $f$ is in the subalgebra of $T_q(n)$ generated by $\{a_{ij}: 2\leq i\leq j\leq n\}$, and is thus a central monomial there. Since this subalgebra is isomorphic to $T_q(n-1)$, it follows by induction that $f=1$, thus proving the claim.
	\end{proof}

	In~\cite{T90}, Takeuchi introduced a bialgebra $M_{\alpha, \beta}$, for $\alpha, \beta\in K^*$, on generators $x_{ij}$, with $1\leq i, j\leq n$. Taking $(\alpha, \beta)=(q^{-1}, q)$, these generators satisfy relations similar to those represented in \cref{E:rels-digram}, so it follows that the $K$-algebra $T_q(n)$ is isomorphic both to the (unital) subalgebra of $M_{q^{-1}, q}$ generated by $\{x_{ij} : 1\leq i\leq j\leq n\}$ and to the quotient $M_{q^{-1}, q}/L$, where $L$ is the two-sided ideal of $M_{q^{-1}, q}$ generated by $\{x_{ij} : j<i\}$. 
	
	The bialgebra structure on $M_{q^{-1}, q}$ is given by the comultiplication $\overline\Delta$ satisfying $\overline\Delta(x_{ij})=\sum_{k=1}^n x_{ik}\otimes x_{kj}$ and the counit $\overline\ve$ such that $\overline\ve(x_{ij})=\dl_{ij}$. Since, for $n\geq 2$, the subalgebra generated by $\{x_{ij} : 1\leq i\leq j\leq n\}$ is not a subcoalgebra, we will need to consider the latter realization of $T_q(n)$ as a quotient of $M_{q^{-1}, q}$. Indeed, since for $j<i$
	\begin{align*}
		\overline\Delta(x_{ij})=\sum_{k=1}^{i-1} x_{ik}\otimes x_{kj} + \sum_{k=i}^n x_{ik}\otimes x_{kj}\in L\otimes M_{q^{-1}, q} + M_{q^{-1}, q}\otimes L
	\end{align*}
	and $\overline\ve(x_{ij})=0$, it follows that $L$ is also a coideal of $M_{q^{-1}, q}$, so $M_{q^{-1}, q}/L$ inherits a bialgebra structure.
	
	Thence, identifying $a_{ij}\in T_q(n)$ with $x_{ij}+L\in M_{q^{-1}, q}/L$, for all $1\le i\le j\le n$, we conclude that $T_q(n)$ is a bialgebra.

	\begin{theorem}
		There are algebra morphisms $\Dl:T_q(n)\to T_q(n)\otimes T_q(n)$ and $\ve:T_q(n)\to K$, given by $\Dl(a_{ij})=\sum_{k=i}^ja_{ik}\otimes a_{kj}$ and $\ve(a_{ij})=\dl_{ij}$, $1\le i\le j\le n$. These define a bialgebra structure on $T_q(n)$.
	\end{theorem}

	\section{The Hopf algebra $UT_q(n)$}
	
	Now, we are going to introduce the analog for upper triangular matrices of the well-known Hopf algebra $GL_q(n)$ (see \cite{RTF89}, \cite[IV.6, IV.10]{Kassel} and references therein).

	\begin{lemma}\label{commut-a_ij-prod-a_xx}
		Let $X\sst [1,n]$ and $1\le i<j\le n$. Then the following equality holds in $T_q(n)$:
		\begin{align}\label{a_ij.prod-a_kk=q^(-2|X-cap-[i_j]|).prod-a_kk.a_ij}
			a_{ij}\left(\prod_{x\in X}a_{xx}\right)=q^m\left(\prod_{x\in X}a_{xx}\right)a_{ij},
		\end{align}
		where $m=-2|\{x\in X: i<x<j\}|-|X\cap\{i,j\}|$.
	\end{lemma}
	\begin{proof}
		We have $a_{ij}a_{ii}=q^{-1}a_{ii}a_{ij}$, $a_{ij}a_{jj}=q^{-1}a_{jj}a_{ij}$, $a_{ij}a_{xx}=q^{-2}a_{xx}a_{ij}$ for all $i<x<j$ and $a_{ij}a_{xx}=a_{xx}a_{ij}$ for all $x\not\in[i,j]$, whence \cref{a_ij.prod-a_kk=q^(-2|X-cap-[i_j]|).prod-a_kk.a_ij}.
	\end{proof}
	
	Set $\dett:=\prod_{i=1}^n a_{ii}\in T_q(n)$.
	
	\begin{corollary}\label{det_a_ij}
		Let $1\le i\le j\le n$. Then 
		\begin{align}\label{det.a_ij=q^(2(j-1))a_ij.det}
			\dett \cdot a_{ij}=q^{2(j-i)}a_{ij}\cdot \dett.
		\end{align}
	\end{corollary}
	
	Referring back to~\cite{T90}, Takeuchi also defines a quantum determinant $g$ in $M_{\alpha, \beta}$ and $\dett$ is precisely the class of $g$ in $M_{q^{-1}, q}/L$. Since $g$ is normal~\cite[p. 214]{Goodearl-Warfield}, it generates an Ore set~\cite[p. 82]{Goodearl-Warfield} in $M_{q^{-1}, q}$. Let $M_{q^{-1}, q}[g^{-1}]$ be the corresponding localization, and similarly for $T_q(n)[\dett^{-1}]$. Both $g$ and $\dett$ are group-like (in fact, $\dett$ is the product of the group-likes $a_{ii}$), so the bialgebra structures on $T_q(n)$ and $M_{q^{-1}, q}$ extend to the respective localizations, with $g^{-1}$ and $\dett^{-1}$ group-like. Moreover, $M_{q^{-1}, q}[g^{-1}]$ is a Hopf algebra, with antipode $\overline S$ defined in~\cite{T90} in terms of quantum minors. Let $L[g^{-1}]$ be the ideal of $M_{q^{-1}, q}[g^{-1}]$ generated by $\{x_{ij} : j<i\}$. Then, it follows that the isomorphism $M_{q^{-1}, q}/L\cong T_q(n)$ extends via localization to a bialgebra isomorphism $M_{q^{-1}, q}[g^{-1}]/L[g^{-1}]\cong T_q(n)[\dett^{-1}]$.
	
	\begin{definition}
		Set $UT_q(n):=T_q(n)[\dett^{-1}]$. Since $T_q(n)$ is a domain, $T_q(n)$ embeds as a subalgebra of $UT_q(n)$. Moreover, since $\dett=\prod_{i=1}^n a_{ii}$, it follows that $UT_q(n)$ can also be seen as the localization of $T_q(n)$ at the multiplicative set generated by $\{a_{ii}\}_{i=1}^n$; in particular, all $a_{ii}$ are invertible in $UT_q(n)$.
	\end{definition}
	
	\begin{theorem}
		The bialgebra $UT_q(n)$ is a Hopf algebra. 
	\end{theorem}
	\begin{proof}
		To prove that $UT_q(n)$ is a Hopf algebra, it suffices to show that $\overline S(x_{ij})\in L[g^{-1}]$ for all $j<i$. This can be readily checked using the fact that every monomial in the quantum minor from the definition of $\overline S(x_{ij})$ in~\cite{T90} contains a factor in $\{x_{ij} : j<i\}$. Hence $M_{q^{-1}, q}[g^{-1}]/L[g^{-1}]$ is a Hopf algebra and using the bialgebra isomorphism $M_{q^{-1}, q}[g^{-1}]/L[g^{-1}]\cong UT_q(n)$, so is $UT_q(n)$. 
	\end{proof}
	
	Note that, in a Hopf algebra, the antipode is the inverse of the identity map relative to the convolution product. We denote the (unique) antipode of $UT_q(n)$ by $S$. Before we give an explicit description of $S$, we want to remark that there is an equivalent construction of $UT_q(n)$ using the skew polynomial algebra $T_q(n)[t;\sg]$, where $\sg$ is an automorphism of the algebra $T_q(n)$ that takes~\cref{det.a_ij=q^(2(j-1))a_ij.det} into account. We begin with a general result on extending a bialgebra structure from a bialgebra $A$ to a skew polynomial algebra $A[t;\sg]$.
	
	\begin{lemma}
		Let $(A,\mu,\eta,\Dl,\ve)$ be a bialgebra and $\sg\in\Aut(A,\mu,\eta)$. Then $\Dl(t)=t\otimes t$ and $\ve(t)=1$ extend the bialgebra structure to the skew polynomial algebra $A[t;\sg]$  if and only if $\sg$ is a bialgebra automorphism of $(A,\mu,\eta,\Dl,\ve)$. 
	\end{lemma}
	\begin{proof}
		It is well-known (see, for example, \cite[Exercise 2.1.19]{Radford}) that the tensor product of two coalgebras $(C,\Dl_C,\ve_C)$ and $(D,\Dl_D,\ve_D)$ is a coalgebra under $\Dl_{C\ot D}=(\id_C\ot\tau_{C,D}\ot\id_D)\circ(\Dl_C\ot\Dl_D)$, where $\tau_{C,D}:C\ot D\to D\ot C$ is the flip map, and $\ve_{C\ot D}=\ve_C\ot \ve_D$. Hence, there is a natural coalgebra structure on the vector space $A[t;\sg]\cong A\otimes K[t]$ (vector space isomorphism!), such that 
		\begin{align}\label{Dl'-and-ve'-on-K[t_sg]}
			\Dl'(at^n)=\sum a_{(1)}t^n\otimes a_{(2)}t^n\text{ and }\ve'(at^n)=\ve(a) 
		\end{align}
		(here we use Sweedler's notation for $\Dl(a)=\sum a_{(1)}\ot a_{(2)}$ and identify $at^n\in A[t;\sg]$ with $a\otimes t^n\in A\otimes K[t]$). Thus, $A[t;\sg]$ is a bialgebra if and only if $\Dl'$ and $\ve'$ are algebra morphisms $A[t;\sg]\to A[t;\sg]\ot A[t;\sg]$ and $A[t;\sg]\to K$, respectively. Observe that \cref{Dl'-and-ve'-on-K[t_sg]} already means that $\Dl'(at)=\Dl'(a)\Dl'(t)$ and $\ve'(at)=\ve'(a)\ve'(t)$ for all $a\in A$. Therefore, $\Dl'$ is an algebra morphism if and only if 
		\begin{align}\label{Dl'(ta)=Dl'(t)Dl'(t)}
			\Dl'(ta)=\Dl'(t)\Dl'(a)    
		\end{align}
		for all $a\in A$. Now, 
		\begin{align*}
			\Dl'(t)\Dl'(a)&=\sum ta_{(1)}\otimes ta_{(2)}=\sum \sg(a_{(1)})t\otimes \sg(a_{(2)})t=((\sg\ot\sg)\circ\Dl)(a)(t\ot t),\\
			\Dl'(ta)&=\Dl'(\sg(a)t)=\sum \sg(a)_{(1)}t\ot \sg(a)_{(2)}t=(\Dl\circ\sg)(a)(t\ot t). 
		\end{align*}
		Thus, \cref{Dl'(ta)=Dl'(t)Dl'(t)} is equivalent to $(\sg\ot\sg)\circ\Dl=\Dl\circ\sg$. Similarly, $\ve'$ is an algebra morphism if and only if $\ve'(ta)=\ve'(t)\ve'(a)$ for all $a\in A$ if and only if $\ve\circ\sg=\ve$.
	\end{proof}

	\begin{lemma}\label{sg-auto-bialg}
		There is a bialgebra automorphism $\sg\in\Aut(T_q(n))$ such that 
		\begin{align}\label{sg(a_ij)=q^(2(i-j))a_ij}
			\sg(a_{ij})=q^{2(i-j)}a_{ij}
		\end{align}
		for all $1\le i\le j\le n$.
	\end{lemma}
	\begin{proof}
		It is clear from~\cref{det.a_ij=q^(2(j-1))a_ij.det} that the right-hand side of \cref{sg(a_ij)=q^(2(i-j))a_ij} is $\dett^{-1}a_{ij}\dett$, so the right conjugation by $\dett$ in $UT_q(n)$ restricts to an algebra automorphism of $T_q(n)$, which we denote by $\sg$. It is a bialgebra automorphism, since $\dett$ is a group-like element.
	\end{proof}
	
	\begin{corollary}
		Let $\sg\in\Aut(T_q(n))$ be given by \cref{sg(a_ij)=q^(2(i-j))a_ij}. Then the bialgebra structure on $T_q(n)$ extends to a bialgebra structure on $T_q(n)[t;\sg]$ by means of $\Dl(t)=t\ot t$ and $\ve(t)=1$.
	\end{corollary}
	
	\begin{remark}\label{det_b} 
		The element $t\dett=\dett t$ is central in $T_q(n)[t;\sg]$.
	\end{remark}

	\begin{lemma}
		The ideal $(t\dett-1)$ of $T_q(n)[t;\sg]$ is also a coideal of the coalgebra $T_q(n)[t;\sg]$. Thus, $T_q(n)[t;\sg]/(t\dett-1)$ is a bialgebra under the comultiplication and counit induced by those from $T_q(n)[t;\sg]$.
	\end{lemma}
	\begin{proof}
		We have 
		\begin{align*}
			\Dl(t\dett-1)&=(t\otimes t)\prod_{i=1}^n(a_{ii}\otimes a_{ii})-1\otimes 1=t\dett\otimes t\dett - 1\otimes 1\\&=t\dett\otimes (t\dett - 1)+(t\dett-1)\otimes 1,
		\end{align*}
		and $\ve(t\dett-1)=1-1=0$.
	\end{proof}

	\begin{proposition}\label{P:UTqn-as-localized-Tqn}
		There is a canonical bialgebra isomorphism $UT_q(n)\to T_q(n)[t;\sg]/(t\dett-1)$ sending $a_{ij}$ to the class of $a_{ij}$ in $T_q(n)[t;\sg]/(t\dett-1)$, for all $1\leq i\leq j\leq n$.
	\end{proposition}
	\begin{proof}
		Although this is straightforward, we will provide the details for the convenience of the reader less familiar with noncommutative localization.  
		
		Let $\Psi:T_q(n)\to T_q(n)[t;\sg]/(t\dett-1)$ be the following composition
		\begin{center}
			\ \ \ \xymatrix{T_q(n) \ar[r]^-\iota & T_q(n)[t;\sg] \ar[r]^-{\pi} & \dfrac{T_q(n)[t;\sg]}{(t\dett-1)},}
		\end{center}
		where $\iota$ is the inclusion and $\pi$ is the canonical epimorphism. So $\Psi(a_{ij})=a_{ij}+I$, where $I=(t\dett-1)$. We have that $\Psi(\dett)=\dett+I$ and $(t+I)(\dett+I)=t\dett +I=1+I=(\dett+I)(t+I)$. Therefore, by the universal property of localization, $\Psi$ extends uniquely to a bialgebra homomorphism $UT_q(n)=T_q(n)[\dett^{-1}]\to T_q(n)[t;\sg]/I$, still denoted by $\Psi$.
		
		For the inverse map, consider the natural inclusion $\Phi: T_q(n)\to UT_q(n)$. For any $b\in T_q(n)$, since $\sg$ is the right conjugation by $\dett$, we have
		\begin{align*}
			\dett^{-1}\Phi(b)= \dett^{-1}b=\sigma(b)\dett^{-1}=\Phi(\sigma(b))\dett^{-1}.
		\end{align*}
		Thus, by the universal property of Ore extensions, $\Phi$ extends to $\Phi: T_q(n)[t;\sg]\to UT_q(n)$ so that $\Phi(t)=\dett^{-1}$. Moreover, $\Phi(t\dett-1)=\dett^{-1}\dett-1=0$, so $\Phi$ factors through a map $T_q(n)[t;\sg]/I\to UT_q(n)$, which is easily seen to be inverse to $\Psi$, yielding the result. 
	\end{proof}
	
	Our final goal is to give an explicit expression for the antipode $S$ on $UT_q(n)$. Define the following elements in $T_q(n)$:

	\begin{align}
		b_{ii}&=\prod_{k\ne i}a_{kk},\ 1\le i\le n,\label{b_ii=prod-a_kk-for-k-ne-i}\\
		b_{ij}&=\sum_{s=1}^{j-i}\sum_{i=i_0<\dots<i_s=j}(-1)^sq^{2(j-i)-s}a_{i_0i_1}\dots a_{i_{s-1}i_s} \prod_{k\not\in\{i_0,\dots,i_s\}}a_{kk},\ 1\le i<j\le n.\label{b_ij=sum-of-prods}
	\end{align}
	
	\begin{remark}\label{R:bij-rec-rel}
		It can be easily checked that the elements $(b_{ij})_{1\le i\le j\le n}$ satisfy the following recurrence relation for $i<j$:
		\begin{align}\label{E:bij-rec-rel}
			b_{ij}=-\sum_{k=i+1}^j q^{2(k-i)-1} a_{ik}b_{kj}a_{ii}^{-1}=-\sum_{k=i+1}^j q^{2(k-i)-1} a_{ik}a_{ii}^{-1}b_{kj},
		\end{align}
		where the terms $b_{kj}a_{ii}^{-1}$ and $a_{ii}^{-1}b_{kj}$ in \cref{E:bij-rec-rel} belong to $T_q(n)$ and are equal since, for $k>i$, $b_{kj}$ is a right and left multiple of $a_{ii}$, $T_q(n)$ is a domain and $a_{ii}$ commutes with $b_{kj}$.
	\end{remark}
	
	Recall the automorphism $\rho$ from \cref{P:auto-rho} and the automorphism $\sg$ from \cref{sg(a_ij)=q^(2(i-j))a_ij}.

	\begin{lemma}\label{L:bij-and-rho}
		We have, for all $1\leq i\leq j\leq n$, $\rho(b_{ij})=b_{n+1-j, n+1-i}$ and $\sigma(b_{ij})=q^{2(i-j)}b_{ij}$. 
	\end{lemma}
	\begin{proof}
		The result holds trivially in the case $i=j$. So, suppose that $i<j$. Notice first that the elements $a_{i_0i_1},\ldots , a_{i_{s-1}i_s}$ appearing in \cref{b_ij=sum-of-prods} pairwise commute, and the same holds for the elements of the form $a_{kk}$. Given $1\leq i\leq n$, set $\overline{i}=n+1-i$. Then we have 
		\begin{align*}
			\rho(b_{ij})&=\sum_{s=1}^{j-i}\sum_{i=i_0<\dots<i_s=j}(-1)^sq^{2(j-i)-s} a_{\overline{i_{s}}\, \overline{i_{s-1}}} \dots a_{\overline{i_1}\, \overline{i_0}}  \prod_{k\not\in\{i_0,\dots,i_s\}}a_{\overline{k}\, \overline{k}}\\
			&=\sum_{s=1}^{\overline i - \overline j}\sum_{\overline j=\overline{i_s}<\dots<\overline{i_0}=\overline{i}}(-1)^sq^{2(\overline i - \overline j)-s} a_{\overline{i_{s}}\, \overline{i_{s-1}}} \dots a_{\overline{i_1}\, \overline{i_0}}  \prod_{\overline{k}\not\in\{\overline{i_s},\dots,\overline{i_0}\}}a_{\overline{k}\, \overline{k}}\\
			&=b_{\overline j \, \overline i}.
		\end{align*}
		The result for $\sg$ is straightforward.
	\end{proof}
	
	The elements $(b_{ij})_{1\le i\le j\le n}$ are the main ingredient for expressing the antipode $S$ of $UT_q(n)$. This will be clearly perceived in the next result.

	\begin{lemma}\label{L:antipode-property}
		For all $1\le i\le j\le n$, we have the following relation in $T_q(n)$:
		\begin{align}\label{sum-b_ik.a_kj=det.dl_ij}
			\sum_{i\le k\le j}b_{ik}a_{kj}=\dett\dl_{ij}=\sum_{i\le k\le j}q^{2(k-j)}a_{ik}b_{kj}.
		\end{align}
	\end{lemma}
	\begin{proof}
		We proceed by induction on $j-i\geq 0$.
		The case $i=j$ is obvious. If $i<j$ then, using \cref{E:bij-rec-rel}, we get
		\begin{align*}
			\sum_{i\le k\le j}b_{ik}a_{kj}&= b_{ii}a_{ij} - \sum_{k=i+1}^j \sum_{l=i+1}^k q^{2(l-i)-1} a_{il}a_{ii}^{-1}b_{lk}a_{kj}\\
			&=b_{ii}a_{ij} - \sum_{l=i+1}^j q^{2(l-i)-1} a_{il}a_{ii}^{-1} \sum_{k=l}^j  b_{lk}a_{kj}\\
			&=b_{ii}a_{ij} - \sum_{l=i+1}^j q^{2(l-i)-1} a_{il}a_{ii}^{-1} \dett\delta_{lj}\\
			&=b_{ii}a_{ij} -  q^{2(j-i)-1} a_{ij}a_{ii}^{-1} \dett\\
			&=b_{ii}a_{ij} -  q^{2(j-i)-1} a_{ij}b_{ii}.
		\end{align*}
		To finish the calculation, notice that 
		\begin{align*}
			a_{ii} (b_{ii}a_{ij} -  q^{2(j-i)-1} a_{ij}b_{ii})
			&=\dett a_{ij} -q^{2(j-i)-1}a_{ii}a_{ij}b_{ii}\\&=q^{2(j-i)}a_{ij}\dett-q^{2(j-i)}a_{ij}a_{ii}b_{ii}=0.
		\end{align*}
		Since $T_q(n)$ is a domain, we conclude that $b_{ii}a_{ij} -  q^{2(j-i)-1} a_{ij}b_{ii}$=0.
		
		The other equality is proved similarly using the recurrence relation
		\begin{align}\label{E:bij-rec-rel-2}
			b_{ij}=-\sum_{l=i}^{j-1} q^{2(j-l)-1} a_{lj}a_{jj}^{-1}b_{il},
		\end{align}
		for $i<j$, which can be obtained from \cref{E:bij-rec-rel} applied to $b_{n+1-j, n+1-i}$ and then using the automorphism $\rho$. 
	\end{proof}

	Recall that a Hopf algebra is pointed if all of its simple left and right comodules have dimension one. Our main result in this section follows.
	
	\begin{theorem}\label{T:sec3:main:thm} 
		The Hopf algebra $UT_q(n)$ is pointed, and its antipode $S$ is given by 
		\begin{align}\label{S(a_ii)=prod+S(a_ij)=prod+S(t)=det}
			S(a_{ij}) =\dett^{-1}b_{ij}=q^{2(i-j)}b_{ij}\dett^{-1}, \text{ for all $1\le i\leq j\le n$,}
		\end{align}
		where $b_{ij}$ is as in~\cref{b_ii=prod-a_kk-for-k-ne-i,b_ij=sum-of-prods}. Moreover, $S^2=\id$.
	\end{theorem}
	\begin{proof}
		The proof that $UT_q(n)$ is pointed is just in~\cite[Proposition 3.1.1]{KO24}.
		The equality $\dett^{-1}b_{ij}=q^{2(i-j)}b_{ij}\dett^{-1}$ follows from~\cref{L:bij-and-rho}, and $S$ has order $2$ because the antipode $\overline S$ defined in~\cite{T90} for $M_{q^{-1}, q}[g^{-1}]$ has this property, and $S$ is induced on $UT_q(n)$ from $\overline S$. 
		
		We show that $S(a_{ij}) =\dett^{-1}b_{ij}$ by induction on $j-i$. For the base case $i=j$ we have, on the one hand, $1=\ve(a_{ii})=S(a_{ii})a_{ii}$, because $S$ is an antipode. On the other hand, by~\cref{L:antipode-property}, we have $\dett^{-1}b_{ii}a_{ii}=\dl_{ii}=1$. Comparing both expressions and canceling $a_{ii}$ on the right, we get $S(a_{ii})=\dett^{-1}b_{ii}$.
		
		Now assume that $j-i>0$ and $S(a_{rs}) =\dett^{-1}b_{rs}$ for all $r\le s$ with $s-r<j-i$. Using again the fact that $S$ is an antipode and the induction hypothesis, we have
		\begin{align*}
			0&=\ve(a_{ij})=\sum_{i\le k\le j}S(a_{ik})a_{kj}=\sum_{i\le k<j}\dett^{-1}b_{ik}a_{kj}+S(a_{ij})a_{jj}.
		\end{align*}
		Now using the identity in \cref{L:antipode-property} and multiplying on the left by $\dett^{-1}$ we obtain
		\begin{align*}
			0=\dl_{ij}=\sum_{i\le k\le j}\dett^{-1}b_{ik}a_{kj}=\sum_{i\le k< j}\dett^{-1}b_{ik}a_{kj} + \dett^{-1}b_{ij}a_{jj}.
		\end{align*}
		Comparing these two expressions and canceling out $a_{jj}$ on the right we deduce that $S(a_{ij})=\dett^{-1}b_{ij}$, which proves the induction step. 
	\end{proof}

	\begin{proposition}\label{autos}
		The automorphisms $\rho, \sg\in\Aut(T_q(n))$ from \cref{P:auto-rho} and \cref{sg-auto-bialg}, respectively, lift uniquely to algebra automorphisms of $UT_q(n)$. These lifts satisfy the following:
		\begin{enumerate}
			\item $\sg$ and $\rho$ commute with the antipode;
			\item $\sg$ is a Hopf algebra automorphism of $UT_q(n)$; 
			\item 
			\label{rho-antiauto-of-coalgebras} $\rho$ is a coalgebra antiautomorphism of $UT_q(n)$.
		\end{enumerate} 
	\end{proposition}
	\begin{proof}
		Since $\sigma(\dett)=\dett=\rho(\dett)$ (where the last equality follows since the elements $a_{ii}$ mutually commute), we conclude that $\sg$ and $\rho$ extend to algebra automorphisms of $UT_q(n)$, with $\sigma(\dett^{-1})=\dett^{-1}$ and $\rho(\dett^{-1})=\dett^{-1}$.
		
		As observed in the proof of \cref{sg-auto-bialg}, $\sg$ is just right conjugation by $\dett$, hence it's a Hopf algebra automorphism of $UT_q(n)$ and it commutes with $S$ because $S(\dett)=\dett^{-1}$.
		
		Given $1\le i\le n$, denote $\overline{i}=n+1-i$. Let $\tau$ be the flip map on $UT_q(n)\ot UT_q(n)$. Then, for all $1\leq i\leq j\leq n$,
		\begin{align*}
			(\rho\otimes\rho)(\Dl(a_{ij}))&=\sum_{i\le k\le j}\rho(a_{ik})\otimes \rho(a_{kj})
			=\sum_{i\le k\le j}a_{\overline k\,\overline i}\otimes a_{\overline j\,\overline k}=\sum_{\overline j\leq\overline k\leq\overline i}a_{\overline k\,\overline i}\otimes a_{\overline j\,\overline k}
			=(\tau\circ\Dl)(a_{\overline j\,\overline i})=(\tau\circ\Dl)(\rho(a_{ij})),\\
			\ve(\rho(a_{ij}))&=\ve(a_{\overline j\, \overline i})=\delta_{\overline j\, \overline i}=\delta_{ij}=\ve(a_{ij}).
		\end{align*}
		Moreover, using \cref{L:bij-and-rho}, we have
		\begin{align*}
			\rho(S(a_{ij}))&=\rho(\dett^{-1}b_{ij})=\dett^{-1}b_{\overline j\,\overline i}=S(a_{\overline j\,\overline i})=S(\rho(a_{ij})).
		\end{align*}
	\end{proof}
	
	\begin{remark}
		Let $X\subseteq \{1, \ldots, n\}$ and consider the unital associative subalgebra of $UT_q(n)$ generated by $\{a_{ii}^{\pm 1} : i\in X\}$. This is just the commutative Laurent polynomial algebra in $k:=|X|$ variables. As $\Dl(a_{ii}^{\pm 1})=a_{ii}^{\pm 1}\ot a_{ii}^{\pm 1}$
		and $S(a_{ii}^{\pm 1})=a_{ii}^{\mp 1}$, this is a Hopf subalgebra. In fact, it is isomorphic to the Hopf group algebra of the free abelian group $\ZZ^k$.
		
		It is straightforward to verify that $(\dett-1)$ is a Hopf ideal of $UT_q(n)$, so we get a quotient Hopf algebra $UT_q(n)/(\dett-1)$. Since
		\begin{align*}
			(q^{2(i-j)}-1)a_{ij}&=(q^{2(i-j)}\dett-1)a_{ij}-q^{2(i-j)}(\dett-1)a_{ij}\\&=a_{ij}(\dett-1)-q^{2(i-j)}(\dett-1)a_{ij}\in (\dett-1),
		\end{align*}
		in case $q$ is not a root of unity, the Hopf algebra $UT_q(n)/(\dett-1)$ is canonically isomorphic to the Hopf subalgebra described above for $X=\{1, \ldots, n-1\}$, isomorphic to the Hopf group algebra of $\ZZ^{n-1}$.
	\end{remark}

	\section{Hopf $\ast$-algebras}
	
	For an involution $\ast$ on a $K$-algebra $A$, we denote $\ast(a)$ by $a^{\ast}$ for each $a\in A$. Recall that a map $\ast: A\to A$ is an \emph{involution} if $(a+b)^\ast=a^\ast+b^\ast$, $(ab)^\ast=b^\ast a^\ast$
	and $(a^\ast)^\ast=a$ for all $a,b\in A$.
	
	Consider the field $K$ of characteristic different from $2$ as a $K$-algebra. Suppose there is an involution $\overline{\phantom{c}} : K\to K$ such that $\overline{\phantom{c}} \neq \id_K$ and let
	$$K_0=\{\alpha\in K : \overline \alpha=\alpha\}.$$
	Since $\,\overline{\phantom{c}}$\, is an automorphism of order $2$, by \cite[Lemma~2.5]{QS}, $K_0$ is a proper subfield of $K$ such that $[K : K_0]=2$ and there is a linear basis $\{1,i\}$ of $K$ over $K_0$ such that $i^2\in K_0$ and $\overline i=-i$.
	
	For any $K$-vector space $V$, a map $\phi:V\to V$ is \emph{antilinear} if $\phi(u+v)=\phi(u)+\phi(v)$ and $\phi(\alpha u)=\overline\alpha \phi(u)$ for all $u,v\in V$ and $\alpha\in K$. Clearly, every antilinear map is $K_0$-linear. An antilinear map between two $K$-algebras $\phi:A\to B$ will be called an \textit{antilinear morphism} if $\phi(ab)=\phi(a)\phi(b)$ for all $a,b\in A$ and $\phi(1)=1$.
	
	\begin{definition}
		Let $(H, \mu, \eta, \Dl,\ve, S)$ be a Hopf $K$-algebra. We say that $H$ is a \emph{Hopf $\ast$-algebra} if there exists an antilinear involution $\ast$ on $H$ satisfying the following conditions:
		\begin{enumerate}
			\item $\ast$ is a morphism of $K_0$-coalgebras;
			\item $(\ast\circ S)^2=\id_H$.
		\end{enumerate} 
	\end{definition}
	
	\begin{remark}
		If $H$ is a Hopf $K$-algebra with an antilinear involution $\ast$, then $\ast$ is an antimorphism of $K_0$-algebras.  
	\end{remark}
	
	\begin{lemma}\label{*-from-gamma}
		A Hopf algebra $H$ has a  Hopf $\ast$-algebra structure if and only if there exists an antilinear bijection $\gamma$ on $H$ such that
		\begin{enumerate}
			\item $\gamma$ is a morphism of $K_0$-algebras and an antimorphism of $K_0$-coalgebras; 
			\item $\gamma^2=(S\circ \gamma)^2=\id_H$.
		\end{enumerate} 
	\end{lemma}
	\begin{proof}
		The proof is similar to that of \cite[Lemma~IV.8.2]{Kassel}.
	\end{proof}
	
	Let $V$ be a $K$-vector space. Then the \textit{conjugate} of $V$ is the $K$-space $\ol V$ which coincides with $(V,+)$ as an abelian group, but is equipped with the following multiplication by scalars: $(\af,v)\mapsto\ol\af v$ for all $\af\in K$ and $v\in V$. Observe that antilinear maps $V\to W$ are exactly the same as linear maps $V\to \ol W$. If $(A,\cdot)$ is a $K$-algebra, then $(\ol A,\cdot)$ is also a $K$-algebra. As a consequence, for any $K$-algebra $R$ there is a one-to-one correspondence between antilinear morphisms $T_q(n)\to R$ and $K$-algebra morphisms $T_q(n)\to \ol R$.

	\begin{proposition}\label{gm-existence}
		Assume that $q\in K_0$. Then there is an antilinear bijection $\gm:UT_q(n)\to UT_q(n)$ which is a morphism of $K_0$-algebras and an antimorphism of $K_0$-coalgebras, such that 
		\begin{align}\label{gm(a_ij)=a_(n+1-j_n+1-i)}
			\gm(a_{ij})=a_{n+1-j,n+1-i}.
		\end{align} 
		Moreover, $\gm^2=\id$.
	\end{proposition}
	\begin{proof}
		As in \cref{P:auto-rho}, since $q\in K_0$, it follows from the observations above that there exists a unique antilinear morphism $\gm:T_q(n)\to UT_q(n)$ mapping $a_{ij}$ to $a_{\overline j\,\overline i}$, where $\overline{i}=n+1-i$. The same argument used in the proof of \cref{autos}\cref{rho-antiauto-of-coalgebras} shows that $\gm$ extends to an antilinear automorphism of the $K$-algebra $UT_q(n)$, which we still denote by $\gm$, such that $\gm(t)=t$. 
		The proof that it is also an antimorphism of $K_0$-coalgebras is the same as the one for $\rho$ in \cref{autos}\cref{rho-antiauto-of-coalgebras}.
		Since $\gm^2(a_{ij})=a_{ij}$ and $\gm^2$ is $K$-linear, then $\gm^2=\id$.
	\end{proof}
	
	\begin{lemma}\label{gm-comm-with-S}
		Assume that $q\in K_0$. Then $\gm\circ S=S\circ\gm$.
	\end{lemma}
	\begin{proof}
		Since both $\gm\circ S$ and $S\circ\gm$ are antilinear, it suffices to show that $(\gm\circ S)(a_{ij})=(S\circ\gm)(a_{ij})$, for all $1\leq i\leq j\leq n$. This follows exactly as for $\rho$ in \cref{autos}\cref{rho-antiauto-of-coalgebras} (based on the computation in \cref{L:bij-and-rho} which holds for $\gm$ because $q\in K_0$), showing that $\rho$ and $S$ commute.
	\end{proof}
	
	\begin{theorem}\label{Hopf-*-str-on-UT_q(n)}
		If $q\in K_0$, then there exists a Hopf $\ast$-algebra structure on $UT_q(n)$ given by $a_{ij}^\ast=(\gm\circ S)(a_{ij})$.
	\end{theorem}
	\begin{proof}
		This follows from \cref{*-from-gamma,gm-existence}, where $(\gm\circ S)^2=\gm^2\circ S^2=\id$ by \cref{gm-comm-with-S,gm-existence,T:sec3:main:thm}. 
	\end{proof}

	\section{Derivations of $T_q(2)$ and $UT_q(2)$}\label{sec-der}

	In this section we assume that $q$ is not a root of unity.
	Since $T_q(2)$ is a quantum affine space (see \cref{ex-T_q(2)}), we can specify the results of~\cite[Theorem 1.2]{Alev-Chamarie} to this algebra.

	For $(s,t)\in\{(1,1),(1,2),(2,2)\}$ and $\nu=(\nu_{11},\nu_{12},\nu_{22})\in\NN^3$ denote by $D_{st,\nu}$ the map $\{a_{11},a_{12},a_{22}\}\to K\gen{a_{11},a_{12},a_{22}}$  
	sending $a_{st}$ to $a^\nu:=a_{11}^{\nu_{11}}a_{12}^{\nu_{12}}a_{22}^{\nu_{22}}$ and $a_{ij}$ to $0$ for $(i,j)\ne (s,t)$. By the Leibniz rule, 
	$D_{st,\nu}$ uniquely extends to a derivation of $K\gen{a_{11},a_{12},a_{22}}$. Moreover, by \cref{a_22a_12-is-qa_12a_22,a_11a_12-is-qa_12a_11,a_11a_22-is-a_22a_11} (see also \cite[p.~1789]{Alev-Chamarie}), $D_{st,\nu}$ defines a derivation of $T_q(2)$ if and only if the following equalities hold in $T_q(2)$:
	\begin{align}
		D_{st,\nu}(a_{22})a_{12}+a_{22}D_{st,\nu}(a_{12})&=q(D_{st,\nu}(a_{12})a_{22}+a_{12}D_{st,\nu}(a_{22})),\label{D(a_22a_12)=qD(a_12a_22)}\\
		D_{st,\nu}(a_{11})a_{12}+a_{11}D_{st,\nu}(a_{12})&=q(D_{st,\nu}(a_{12})a_{11}+a_{12}D_{st,\nu}(a_{11})),\label{D(a_11a_12)=qD(a_12a_11)}\\
		D_{st,\nu}(a_{11})a_{22}+a_{11}D_{st,\nu}(a_{22})&=D_{st,\nu}(a_{22})a_{11}+a_{22}D_{st,\nu}(a_{11}).\label{D(a_11a_22)=D(a_22a_11)}
	\end{align} 
	The following lemma characterizes this situation.
	
	\begin{lemma}\label{dertypes}
		We have
		\begin{enumerate}
			\item\label{D_1_nu} $D_{11,\nu}$ defines a derivation of $T_q(2)$ if and only if $\nu\in\{(0,0,1),(1,0,0)\}$;
			\item\label{D_2_nu} $D_{12,\nu}$ defines a derivation of $T_q(2)$ if and only if $\nu_{12}=1$;
			\item\label{D_3_nu} $D_{22,\nu}$ defines a derivation of $T_q(2)$ if and only if $\nu\in\{(0,0,1),(1,0,0)\}$.
		\end{enumerate}
	\end{lemma}
	\begin{proof}
		\textit{\cref{D_1_nu}} Let $(s,t)=(1,1)$. Observe that \cref{D(a_11a_12)=qD(a_12a_11)} is equivalent to $D_{11,\nu}(a_{11})a_{12}=q a_{12}D_{11,\nu}(a_{11})$, i.e., $a^\nu\cdot a_{12}=qa_{12}\cdot a^\nu\iff (q^{\nu_{22}}-q^{1-\nu_{11}})a_{11}^{\nu_{11}}a_{12}^{\nu_{12}+1}a_{22}^{\nu_{22}}=0$. Since $q$ is not a root of unity, the latter is equivalent to $\nu_{11}+\nu_{22}=1\iff\{\nu_{11},\nu_{22}\}=\{0,1\}$. Furthermore, \cref{D(a_11a_22)=D(a_22a_11)} reduces to $D_{11,\nu}(a_{11})a_{22}=a_{22}D_{11,\nu}(a_{11})\iff a^\nu\cdot a_{22}=a_{22}\cdot a^\nu\iff (1-q^{\nu_{12}})a_{11}^{\nu_{11}}a_{12}^{\nu_{12}}a_{22}^{\nu_{22}+1}=0\iff\nu_{12}=0$. Finally, \cref{D(a_22a_12)=qD(a_12a_22)} holds trivially for any $\nu$.
		
		\textit{\cref{D_2_nu}} Let $(s,t)=(1,2)$. We have \cref{D(a_11a_12)=qD(a_12a_11)} $\iff a_{11}D_{12,\nu}(a_{12})=qD_{12,\nu}(a_{12})a_{11}\iff a_{11}\cdot a^\nu=qa^\nu\cdot a_{11}\iff (1-q^{1-\nu_{12}})a_{11}^{\nu_{11}+1}a_{12}^{\nu_{12}}a_{22}^{\nu_{22}}=0$, \cref{D(a_22a_12)=qD(a_12a_22)} $\iff a_{22}D_{12,\nu}(a_{12})=qD_{12,\nu}(a_{12})a_{22}\iff a_{22}\cdot a^\nu=qa^\nu\cdot a_{22}\iff (q^{\nu_{12}}-q)a_{11}^{\nu_{11}}a_{12}^{\nu_{12}}a_{22}^{\nu_{22}+1}=0$ and \cref{D(a_11a_22)=D(a_22a_11)} is trivially satisfied, giving the unique condition $\nu_{12}=1$.
		
		\textit{\cref{D_3_nu}} Let $(s,t)=(2,2)$. We have \cref{D(a_11a_22)=D(a_22a_11)} $\iff a_{11}D_{22,\nu}(a_{22})=D_{22,\nu}(a_{22})a_{11}\iff a_{11}\cdot a^\nu=a^\nu\cdot a_{11}\iff (1-q^{-\nu_{12}})a_{11}^{\nu_{11}+1}a_{12}^{\nu_{12}}a_{22}^{\nu_{22}}=0$, \cref{D(a_22a_12)=qD(a_12a_22)} $\iff D_{22,\nu}(a_{22})a_{12}=qa_{12}D_{22,\nu}(a_{22})\iff a^\nu\cdot a_{12}=qa_{12} \cdot a^\nu\iff (q^{\nu_{22}}-q^{1-\nu_{11}})a_{11}^{\nu_{11}}a_{12}^{\nu_{12}+1}a_{22}^{\nu_{22}}=0$ and \cref{D(a_11a_12)=qD(a_12a_11)} is trivially satisfied, giving $\nu_{12}=0$ and $\nu_{11}+\nu_{22}=1$.
	\end{proof}
	
	Following~\cite[1.2]{Alev-Chamarie}, set 
	\begin{align*}
		D_{11}:=D_{11,(1,0,0)},\quad D_{12}:=D_{12,(0,1,0)},\quad D_{22}:=D_{22,(0,0,1)}.
	\end{align*} 
	Furthermore, if 
	\begin{align*}
		\Lb_{st}=\{\nu\in\NN^3: \nu_{st}=0\text{ and }D_{st,\nu}\in\Der(T_q(2))\},    
	\end{align*}
	then \cref{dertypes} gives
	\begin{align*}
		\Lb_{11}=\{(0,0,1)\},\quad \Lb_{12}=\emptyset,\quad \Lb_{22}=\{(1,0,0)\},   
	\end{align*}
	so that
	\begin{align}\label{E:E-for-Tq2}
		E:=\Span_K\{D_{st,\nu}: 1\le s\le t\le 2\text{ and }\nu\in\Lb_{st}\}=K D_{11,(0,0,1)}\oplus K D_{22,(1,0,0)}.
	\end{align}

	By~\cite[Theorem 1.2]{Alev-Chamarie} and \cref{center-T_q(n)} we have the following characterization of $\Der(T_q(2))$.
	
	\begin{corollary}\label{D_ij-and-E-for-T_q(2)}
		\begin{align}\label{Der(T_q(2))-decomp}
			\Der(T_q(2))=\IDer(T_q(2))\oplus KD_{11}\oplus K D_{12}\oplus K D_{22}\oplus K D_{11,(0,0,1)}\oplus K D_{22,(1,0,0)}.
		\end{align}
		In particular, $\dim(\hoch^1(T_q(2)))=5$.
	\end{corollary}

	Next, we tackle the derivation Lie algebra of the Hopf algebra $UT_q(2)$. We will see that $\hoch^1(UT_q(2))$ is infinite dimensional, although free of rank $3$ over $\hoch^0(UT_q(2))=Z(UT_q(2))$.
	
	Recall that $UT_q(2)$ can be seen as the localization of $T_q(2)$ at the powers of $a_{11}$ and $a_{22}$. We will consider the subalgebra $\mathcal P_q$ generated by $a_{11}^{\pm1}$ and $a_{12}$, and the localization $\mathcal T_q$ of $\mathcal P_q$ at the powers of $a_{12}$. Notice that both $\mathcal T_q$ and $\mathcal P_q$ are localizations of the quantum plane $A_2(q\m)$ generated by $a_{11}$ and $a_{12}$. 
	For simplicity, we use the following more intuitive notation: $\mathcal P_q=K_q[a_{11}^{\pm1}, a_{12}]$ and $\mathcal T_q=K_q[a_{11}^{\pm1}, a_{12}^{\pm1}]$.
	
	\begin{lemma}\label{interm-centers}
		We have $Z(\mathcal T_q)=Z(\mathcal P_q)=K$ and $Z(UT_q(2))=K[z^{\pm 1}]$, where $z=a_{11}a_{22}^{-1}$. Moreover, $UT_q(2)=\mathcal P_q[z^{\pm 1}]$, a (commutative) Laurent polynomial extension of $\mathcal P_q$.
	\end{lemma}
	\begin{proof}
		The commutation relations for $T_q(2)$ and $UT_q(2)$ are given by the matrix $Q$ in \cref{ex-T_q(2)}. The corresponding matrix of exponents is $\mathcal Q=\left(\begin{smallmatrix} 0&1&0\\{-1}&0&{-1}\\ 0&1&0\end{smallmatrix}\right)$ and $Z(UT_q(2))$ is the linear span of the monomials $a_{11}^{\nu_{11}}a_{12}^{\nu_{12}}a_{22}^{\nu_{22}}$ with $(\nu_{11}, \nu_{12}, \nu_{22})\in \ZZ\times\NN\times\ZZ$ in the null space of $\mathcal Q$. Since this null space in $\ZZ^3$ is $\{k(1,0,-1): k\in\ZZ\}$, it follows that $Z(UT_q(2))$ is the Laurent polynomial ring $K[z^{\pm 1}]$. Similarly, since the matrix of exponents of the commutation relations for $\mathcal T_q$ and $\mathcal P_q$ is $\left(\begin{smallmatrix} 0&1\\{-1}&0\end{smallmatrix}\right)$, which has trivial null space, it follows that $Z(\mathcal T_q)=Z(\mathcal P_q)=K$.
		
		Clearly, $\mathcal P_q[z^{\pm 1}]\subseteq UT_q(2)$. Note that $\mathcal P_q[z^{\pm 1}]$ is generated as a vector space by $\{ a_{11}^{i}a_{12}^{j}a_{11}^{k}a_{22}^{-k}: i, k\in\ZZ, j\in\NN \}$. Up to nonzero scalars, this set is just $\{ a_{11}^{i}a_{12}^{j}a_{22}^{k}: i, k\in\ZZ, j\in\NN \}$, which is linearly independent, hence a basis of $\mathcal P_q[z^{\pm 1}]$. But the latter is also a basis of $UT_q(2)$, showing the equality $UT_q(2)=\mathcal P_q[z^{\pm 1}]$ and the algebraic independence of $z^{\pm 1}$ over $\mathcal P_q$. Note in particular that $a_{22}^{\pm1}=a_{11}^{\pm1}z^{\mp1}$.
	\end{proof}
	
	Note that any derivation $D$ of $T_q(2)$ extends uniquely to a derivation of $UT_q(2)$, still denoted $D$, by localization, using $D(x^{-1})=-x^{-1} D(x)x^{-1}$, for $x=a_{11}^{k}, a_{22}^{k}$ ($k\in\NN$) and the Leibniz rule. In particular, the (extended) derivations $D_{11}$, $D_{12}$ and $D_{22}$ play an important role in the description of $\Der(UT_q(2))$.
	
	\begin{theorem}
		We have
		\begin{align}\label{Der(UT_q(2))-decomp}
			\Der(UT_q(2))=\IDer(UT_q(2))\oplus K[z^{\pm 1}] D_{11}\oplus K[z^{\pm 1}] D_{12}\oplus K[z^{\pm 1}] D_{22}.
		\end{align}
		In particular, $\hoch^1(UT_q(2))$ is a free module of rank $3$ over $Z(UT_q(2))=K[z^{\pm 1}]$.
	\end{theorem}
	\begin{proof}
		We begin by computing and decomposing the derivation Lie algebra of $\mathcal P_q=K_q[a_{11}^{\pm1}, a_{12}]$. Since $Z(\mathcal T_q)=K$, it follows from \cite[Corollary 2.6]{OP95} that 
		\begin{align*}
			\Der(\mathcal P_q)=\IDer(\mathcal P_q)\oplus K d_{11}\oplus K d_{12},
		\end{align*}
		where $d_{11}(a_{11})=a_{11}$, $d_{11}(a_{12})=0$, $d_{12}(a_{11})=0$ and $d_{12}(a_{12})=a_{12}$. 
		
		Now, using the equality $UT_q(2)=\mathcal P_q[z^{\pm 1}]$ from \cref{interm-centers} and \cite[Theorem 2.1]{LLO25pr} (where we need to use an analogous version for central Laurent extensions, instead of polynomial extensions, but which has the same proof), we can conclude that 
		\begin{align*}
			\Der(\mathcal P_q[z^{\pm 1}])=\IDer(\mathcal P_q[z^{\pm 1}])\oplus K[z^{\pm 1}] \overline{d_{11}}\oplus K[z^{\pm 1}] \overline{d_{12}}\oplus K[z^{\pm 1}]\partial_z,
		\end{align*}
		where $\overline{d_{11}}$ and $\overline{d_{12}}$ are $d_{11}$ and $d_{12}$ extended to $\mathcal P_q[z^{\pm 1}]$ by setting $\overline{d_{11}}(z)=0=\overline{d_{12}}(z)$, $\partial_z(\mathcal P_q)=0$ and $\partial_z(z)=1$. 
		
		Using the fact that $a_{22}=a_{11}z^{-1}$, we can compute the values of these three derivations at $a_{22}$:
		\begin{align*}
			\overline{d_{11}}(a_{22})&=\overline{d_{11}}(a_{11}z^{-1})=\overline{d_{11}}(a_{11})z^{-1} + a_{11}\overline{d_{11}}(z^{-1}) =a_{11}z^{-1}=a_{22},\\
			\overline{d_{12}}(a_{22})&=\overline{d_{12}}(a_{11}z^{-1})=\overline{d_{12}}(a_{11})z^{-1} + a_{11}\overline{d_{12}}(z^{-1}) =0,\\
			\partial_z(a_{22})&=\partial_z(a_{11}z^{-1})=\partial_z(a_{11})z^{-1} + a_{11}\partial_z(z^{-1}) =-a_{11}z^{-2}.
		\end{align*}
		The next table summarizes all the relevant information:
		\begin{center}
			\begin{tabular}{c|ccc}
				&$\overline{d_{11}}$& $\overline{d_{12}}$ & $\partial_z$\\ \hline
				$a_{11}$&$a_{11}$ &0 &0 \\
				$a_{12}$& 0& $a_{12}$&0 \\
				$a_{22}$&$a_{22}$ &0 & $-z^{-2}a_{11}$
			\end{tabular}
		\end{center}
		It follows that $\overline{d_{11}}=D_{11}+D_{22}$, $\overline{d_{12}}=D_{12}$ and $\partial_z=-z^{-1}D_{22}$. Conversely, $D_{11}=\overline{d_{11}}+z\partial_z$, $D_{12}=\overline{d_{12}}$ and $D_{22}=-z\partial_z$, so 
		\[
		K[z^{\pm 1}] \overline{d_{11}}\oplus K[z^{\pm 1}] \overline{d_{12}}\oplus K[z^{\pm 1}]\partial_z=K[z^{\pm 1}] D_{11}\oplus K[z^{\pm 1}] D_{12}\oplus K[z^{\pm 1}] D_{22}.
		\]
	\end{proof}
	
	\section{Automorphisms of $T_q(2)$ and $UT_q(2)$}\label{sec-auto-T_q(2)}

	In this final section, we will give a complete description of the groups of algebra automorphisms of $T_q(2)$ and of $UT_q(2)$ in the case that $q$ is not a root of unity.

	\begin{remark}\label{vf-determined-by-p_ij}
		Any $\vf\in \End(T_q(2))$ is uniquely determined by $\vf(a_{ij})=p_{ij}$ satisfying 
		\begin{align}
			p_{22}p_{12}&=qp_{12}p_{22},\label{f_22f_12-is-qf_12f_22}\\
			p_{11}p_{12}&=qp_{12}p_{11},\label{f_11f_12-is-qf_12f_11}\\
			p_{11}p_{22}&=p_{22}p_{11}.\label{f_11f_22-is-f_22f_11}
		\end{align}
		Moreover, in case $\vf\in \End(UT_q(2))$, we additionally need to require that $p_{11}$ and $p_{22}$ be invertible in $UT_q(2)$, thus (nonzero scalar multiples of) monomials in $a_{11}$ and $a_{22}$.
	\end{remark}
	
	In particular, note that for any choice of $\alpha_{11}, \alpha_{12}, \alpha_{22}\in K^*$, there are automorphisms of $T_q(2)$ and of $UT_q(2)$ such that $a_{ij}\mapsto \alpha_{ij}a_{ij}$, for all $1\leq i\leq j\leq 2$. We call these the \textit{diagonal} automorphisms. The set of all such automorphisms forms a group isomorphic to $(K^*)^3$. We also recall the automorphism $\rho$ defined in \cref{P:auto-rho}, which fixes $a_{12}$ and interchanges $a_{11}$ and $a_{22}$. Since, in general, $\rho$ does not commute with the diagonal automorphisms, it follows that $\Aut(T_q(2))$ and $\Aut(UT_q(2))$ are nonabelian.

	\begin{lemma}\label{L:a12-fixed-by-autos}
		Suppose that $q\neq 1$. Let $A$ be either $T_q(2)$ or $UT_q(2)$ and $\vf\in\Aut(A)$. Then there is a unit $u\in A$ such that $\vf(a_{12})=ua_{12}$.
	\end{lemma}
	\begin{proof}
		Consider $I=(a_{12})$ the ideal of $A$ generated by $a_{12}$. We have $T_q(2)/I\cong K[a_{11}, a_{22}]$ (respectively, $UT_q(2)/I\cong K[a^{\pm1}_{11}, a^{\pm1}_{22}]$), a commutative polynomial ring (respectively, Laurent polynomial ring) in two variables. In particular, $A/I$ is a commutative domain of Gelfand--Kirillov dimension $2$.

		The automorphism $\vf$ induces an isomorphism between $A/I$ and $A/\vf(I)$,
		hence $A/\vf(I)$ is also a commutative domain of Gelfand--Kirillov dimension $2$. In particular, in $A/\vf(I)$ we have
		\begin{align*}
			\overline 0= [\ol {a_{12}},\ol{a_{11}}]=\overline{[a_{12},a_{11}]}=(1-q)\overline{a_{12}a_{11}}=(1-q)\ol{a_{12}}\;\ol{a_{11}}.
		\end{align*}
		As $q\neq 1$ and $A/\vf(I)$ is a domain, we conclude that either $a_{12}\in\vf(I)$ or $a_{11}\in\vf(I)$. A similar computation shows that either $a_{12}\in\vf(I)$ or $a_{22}\in\vf(I)$. As $a_{11}$ and $a_{22}$ are units in $UT_q(2)$, it immediately follows that $a_{12}\in\vf(I)$ in case $A=UT_q(2)$. So assume that $A=T_q(2)$. If $a_{12}\notin\vf(I)$, we must have $a_{11}, a_{22}\in\vf(I)$. Hence $T_q(2)/\vf(I)$ is isomorphic to a quotient of $T_q(2)/(a_{11}, a_{22})\cong K[a_{12}]$. By \cite[Lemma 3.1]{Krause-Lenagan} this would imply that the Gelfand--Kirillov dimension of $T_q(2)/\vf(I)$ is at most $1$, which is a contradiction. Thus, $a_{12}\in\vf(I)$.
		
		We now have $I=(a_{12})\subseteq \vf(I)$. Since $\vf$ was arbitrary, the same holds for $\vf\m$, whence $I\sst\vf\m(I)$.
		Thus $(a_{12})= I=\vf(I)=(\vf(a_{12}))$. As $a_{12}$ is normal in $A$ (i.e., $a_{12}A=Aa_{12}$), we are done.
	\end{proof}
	
	We are now ready to describe $\Aut(T_q(2))$. The following notion will be helpful.
	An automorphism $\vf\in\Aut(T_q(2))$ is called \textit{linear} if 
	\begin{align*}
		\vf(\Span_K\{a_{11},a_{12},a_{22}\})=\Span_K\{a_{11},a_{12},a_{22}\}.
	\end{align*}

	\begin{theorem}\label{T:aut:Tq2}
		Suppose that $q$ is not a root of unity. Then any automorphism of $T_q(2)$ is linear. Moreover, 
		\begin{equation*}
			\Aut(T_q(2))\cong GL_1(K)\times GL_2(K)=K^*\times GL_2(K).
		\end{equation*}
	\end{theorem}
	\begin{proof}
		Let $\lambda\in GL_1(K)=K^*$ and $A=(\mu_{ij})\in GL_2(K)$. The pair $(\lambda, A)\in GL_1(K)\times GL_2(K)$ corresponds to the assignment $p_{12}=\lambda a_{12}$, $p_{11}=\mu_{11}a_{11}+\mu_{21}a_{22}$ and $p_{22}=\mu_{12}a_{11}+\mu_{22}a_{22}$. It is immediate to check that the elements $p_{ij}$ satisfy \cref{f_22f_12-is-qf_12f_22,f_11f_12-is-qf_12f_11,f_11f_22-is-f_22f_11}, so they define $\vf_{\lambda, A}\in \End(T_q(2))$ with $\vf_{\lb, A}(a_{ij})=p_{ij}$. Clearly, $\vf_{\lambda, A}$ is invertible, with inverse $\vf_{\lambda^{-1}, A^{-1}}$. This proves that $GL_1(K)\times GL_2(K)$ embeds naturally in the group of linear automorphisms of $T_q(2)$.
		
		Conversely, let $\vf\in\Aut(T_q(2))$. Note that both $\rho$ and the diagonal automorphisms defined earlier are linear, so we can work modulo these automorphisms. 
		
		Since the group of units of $T_q(2)$ is reduced to scalars, we conclude from \cref{L:a12-fixed-by-autos} that $\vf(a_{12})=\lambda a_{12}$, for some $\lambda\in K^*$. Composing with an appropriate diagonal automorphism, we can suppose that $\lambda=1$, so that $\vf$ fixes $a_{12}$.
		
		Write $\vf(a_{11})=\sum_{k, l\geq 0}g_{kl}(a_{12})a_{11}^k a_{22}^l$, for some polynomials $g_{kl}(y)\in K[y]$. Using \cref{f_11f_12-is-qf_12f_11} we obtain
		\begin{align*}
			\sum_{k, l\geq 0}q^{k+l}a_{12}g_{kl}(a_{12})a_{11}^k a_{22}^l&=\left(\sum_{k, l\geq 0}g_{kl}(a_{12})a_{11}^k a_{22}^l\right) a_{12}= q a_{12}\left(\sum_{k, l\geq 0}g_{kl}(a_{12})a_{11}^k a_{22}^l\right)\\&=q\sum_{k, l\geq 0}a_{12}g_{kl}(a_{12})a_{11}^k a_{22}^l.
		\end{align*}
		As $q$ is not a root of unity, 
		we conclude that $g_{kl}(y)=0$, whenever $k+l\ne 1$. Thus, we can write $\vf(a_{11})=g_{10}(a_{12})a_{11} +g_{01}(a_{12}) a_{22}$. Similarly, $\vf(a_{22})=h_{10}(a_{12})a_{11} +h_{01}(a_{12}) a_{22}$, for some $h_{10}(y), h_{01}(y)\in K[y]$. 
		
		Next we use \cref{f_11f_22-is-f_22f_11}. Note that $a_{ii}p(a_{12})=p(qa_{12})a_{ii}$ for $i=1, 2$ and all $p(y)\in K[y]$. Thus, the (left) coefficient (in $K[a_{12}]$) of $a_{11}^2$ in the product $\vf(a_{11})\vf(a_{22})$ is $g_{10}(a_{12})h_{10}(qa_{12})$. Similarly, the coefficient of $a_{11}^2$ in the product $\vf(a_{22})\vf(a_{11})$ is $h_{10}(a_{12})g_{10}(qa_{12})$. Equating these terms we obtain
		\begin{align}\label{E:ghq-eq-hgq}
			g_{10}(a_{12})h_{10}(qa_{12})=h_{10}(a_{12})g_{10}(qa_{12}).
		\end{align}
		
		\textit{Claim:} Let $g, h\in K[y]$ satisfy
		\begin{align}\label{g(t)h(qt)=h(t)g(qt)}
			g(y)h(qy)=h(y)g(qy).    
		\end{align}
		Then one of $g$ or $h$ is a scalar multiple of the other.
		
		\textit{Proof of the Claim:} Without loss of generality, we can assume that $g, h\neq 0$ and that they are monic. Moreover, if $d\in K[y]$ is a common divisor of $g$ and $h$, then we can replace $g$ and $h$ by $g/d$ and $h/d$ in \cref{g(t)h(qt)=h(t)g(qt)}, so we can assume that $g$ and $h$ are coprime. Then, as $g(y)$ divides $h(y)g(qy)$ and is coprime to $h(y)$, $g(y)$ must divide $g(qy)$. For degree considerations, $g(y)=\mu g(qy)$, for some $\mu\in K^*$. But the fact that $q$ is not a root of unity implies that $g$ is a monomial in $y$. The same argument applies to $h$. By coprimeness, one of these polynomials is $1$ and the other $y^{m}$, for some $m\in\NN$. Then using  \cref{g(t)h(qt)=h(t)g(qt)} we conclude that $q^{m}=1$, so also $m=0$. This completes the proof of the claim.
		\medskip
		
		Applying the Claim to $g_{10}$ and $h_{10}$, by \cref{E:ghq-eq-hgq} we conclude that $g_{10}$ and $h_{10}$ are proportional. Similarly, comparing (left) coefficients of $a_{22}^2$ in \cref{f_11f_22-is-f_22f_11}, we conclude that $g_{01}$ and $h_{01}$ are also proportional. Up to composing with $\rho$, we can assume that $g_{10}\neq 0$. So $h_{10}=\lambda g_{10}$, for some $\lambda\in K$. Also, $g_{01}$ and $h_{01}$ cannot both be $0$ (otherwise $\vf(a_{22})=\lb\vf(a_{11})$), so assume $h_{01}\neq 0$ (the case $g_{01}\ne 0$ is similar), whence $g_{01}=\mu h_{01}$, for some $\mu\in K$.
		
		Since $\vf^{-1}$ fixes $a_{12}$ and acts on $a_{ii}$, $i=1,2$, similarly to $\vf$, it follows from $\vf(\vf^{-1}(a_{ii}))=a_{ii}$, $i=1,2$, that the matrix $A=\left(\begin{smallmatrix} g_{10}&h_{10}\\g_{01}&h_{01}\end{smallmatrix}\right)$ is invertible over $K[a_{12}]$. So, $\det(A)\in K^*$. By the previous paragraph, $\det(A)$ is a multiple of $g_{10}h_{01}$, which implies that $g_{10}, h_{01}\in K^*$ and thus $A\in GL_2(K)$ and $\vf=\vf_{\lb,A}$, as needed.
	\end{proof}
	
	\begin{remark}
		Observe that $\vf\in\Aut(T_q(2))$ is a bialgebra automorphism if and only if $\vf(a_{12})=\lb a_{12}$ and $\vf(a_{ii})=a_{ii}$, $i=1,2$. Thus, the group of bialgebra automorphisms of $T_q(2)$ is isomorphic to $K^*$.
	\end{remark}
	
	Next, we tackle the automorphism group of $UT_q(2)$. Although the group of units of $UT_q(2)$ is larger, there is a bit more rigidity in $\Aut(UT_q(2))$ in the sense that, modulo the subgroup $\langle\rho \rangle\cong\ZZ/2\ZZ$ generated by $\rho$, the automorphisms of $UT_q(2)$ are in a certain sense diagonal (see below).
	
	\begin{theorem}\label{Aut(UT_q(2))-cong-G-rtimes-<rho>}
		The following is a subgroup of $\Aut(UT_q(2))$:
		\begin{align*}
			\mathcal{G}=\{\vf\in\End(UT_q(2)) : \vf(a_{12})=\lambda_{12}a_{11}^ka_{22}^l a_{12},\ \vf(a_{ii})=\lambda_{ii}z^ja_{ii},\ i\in\{1, 2\}, \lambda_{ij}\in K^*, j, k, l\in\ZZ \},
		\end{align*}
		where $z=a_{11}a_{22}^{-1}$. Moreover, if $q$ is not a root of unity, then
		\begin{align*}
			\Aut(UT_q(2))=\mathcal{G}\rtimes\langle\rho\rangle.
		\end{align*}
	\end{theorem}
	\begin{proof}
		Using \cref{f_22f_12-is-qf_12f_22,f_11f_12-is-qf_12f_11,f_11f_22-is-f_22f_11}, it is routine to check that $\mathcal{G}\subseteq \End(UT_q(2))$ and that $\mathcal{G}$ is a submonoid of $\End(UT_q(2))$. Moreover, it is not hard to check that in fact $\mathcal{G}\subseteq \Aut(UT_q(2))$ is a subgroup. This can also be done explicitly as follows. Identify $\mathcal{G}$ as a set with $(K^*)^3\times\ZZ^3$ (note that these are not isomorphic as groups, because $\mathcal{G}$ is nonabelian), where $\vf\in \mathcal{G}$ with $\vf(a_{12})=\lambda_{12}a_{11}^ka_{22}^l a_{12}$ and $\vf(a_{ii})=\lambda_{ii}z^ja_{ii}$, for $i=1, 2$, is identified with $(\lambda_{12}, \lambda_{11}, \lambda_{22}, j, k, l)\in (K^*)^3\times\ZZ^3$. Then, by transport of structure, the operation induced on $(K^*)^3\times\ZZ^3$ via composition in $\mathcal{G}$ is given by
		\begin{align}
			&(\lambda_{12}, \lambda_{11}, \lambda_{22}, j_1, k_1, l_1)\ast (\mu_{12}, \mu_{11}, \mu_{22}, j_2, k_2, l_2)\notag\\ 
			&=\left(\lambda_{12}\mu_{12}\lambda_{11}^{k_2}\lambda_{22}^{l_2}, \lambda_{11}\mu_{11}\left(\frac{\lambda_{11}}{\lambda_{22}}\right)^{j_2}, \lambda_{22}\mu_{22}\left(\frac{\lambda_{11}}{\lambda_{22}}\right)^{j_2}, j_1+j_2, k_1+k_2+j_1(k_2+l_2), l_1+l_2-j_1(k_2+l_2)\right).\label{product-of-sextuples} 
		\end{align}
		Thus the identity morphism corresponds to $(1, 1, 1, 0, 0, 0)$ and $(\lambda_{12}, \lambda_{11}, \lambda_{22}, j, k, l)^{-1}$ is
		\begin{align*}
			\left(\lambda_{12}^{-1}\lambda_{11}^{k-j(k+l)}\lambda_{22}^{l+j(k+l)}, \lambda_{11}^{-1}\left(\frac{\lambda_{11}}{\lambda_{22}}\right)^{j}, \lambda_{22}^{-1}\left(\frac{\lambda_{11}}{\lambda_{22}}\right)^{j}, -j, j(k+l)-k, -j(k+l)-l\right). 
		\end{align*}
		Moreover, if $\vf \in \mathcal{G}$ is represented by $(\lambda_{12}, \lambda_{11}, \lambda_{22}, j, k, l)$, then $\rho\circ\vf\circ\rho$ is represented by $(\lambda_{12}, \lambda_{22}, \lambda_{11}, -j, l, k)$, so $\rho$ normalizes $\mathcal{G}$ and clearly $\mathcal{G}\cap\langle\rho\rangle=\id_{UT_q(2)}$.
		
		Assume now that $q$ is not a root of unity. It remains to show that $\Aut(UT_q(2))$ is generated by $\mathcal{G}$ and $\rho$. Note that $\mathcal{G}$ contains the diagonal automorphisms of $UT_q(2)$, so we can work modulo $\rho$ and the diagonal automorphisms. Let $\vf\in\Aut(UT_q(2))$. Since the group of units of $UT_q(2)$ is $\{\lambda a_{11}^ka_{22}^l : \lambda\in K^*, k, l\in\ZZ \}$, we conclude from \cref{L:a12-fixed-by-autos} that, modulo a diagonal automorphism, $\vf(a_{12})= a_{11}^ka_{22}^l a_{12}$, for some $k, l\in\ZZ$. 
		
		Since $a_{11}$ is a unit, there exist $i, j\in\ZZ$ and $\lambda\in K^*$ such that $\vf(a_{11})=\lambda a_{11}^i a_{22}^{j}$; as before, we can assume that $\lambda=1$. Using the relation \cref{f_11f_12-is-qf_12f_11} and the fact that $q$ is not a root of unity, we conclude, as in the proof of \cref{T:aut:Tq2}, that $i+j=1$, so 
		$\vf(a_{11})= a_{11}^i a_{22}^{1-i}=a_{11}z^{i-1}$. Similarly, we have $\vf(a_{22})=a_{11}^j a_{22}^{1-j}=a_{22}z^{j}$, for some $j\in\ZZ$. The automorphism $\vf$ must restrict to an automorphism of $Z(UT_q(2))$ and $Z(UT_q(2))=K[z^{\pm 1}]$, by \cref{interm-centers}. We have
		\begin{align*}
			\vf(z)=\vf(a_{11}a_{22}^{-1}) = a_{11}z^{i-1}a_{22}^{-1}z^{-j}=z^{i-j},
		\end{align*}
		whence $i-j=\pm 1$. Interchanging $\vf$ with $\vf\circ\rho$, if necessary, we can assume that $i-j= 1$. Thus $\vf(a_{11})=a_{11}z^{j}$ and $\vf(a_{22})=a_{22}z^{j}$, so $\vf\in \mathcal{G}$.
	\end{proof}

	\begin{remark}
		Let $\mathcal{G}$ be the subgroup of $\Aut(UT_q(2))$ from \cref{Aut(UT_q(2))-cong-G-rtimes-<rho>}. Then each $(\lb_{12},\lb_{11},\lb_{22},j,k,l)\in \mathcal{G}$ can be uniquely written as
		\begin{align}
			(\lb_{12},\lb_{11},\lb_{22},j,k,l)&=(1,1,1,j,k,l)*(\lb_{12},\lb_{11},\lb_{22},0,0,0)\notag\\
			&=(1,1,1,0,k,l)*(1,1,1,j,0,0)*(\lb_{12},\lb_{11},\lb_{22},0,0,0)\notag\\
			&=(1,1,1,j,0,0)*(1,1,1,0,k-j(k+l),l+j(k+l))*(\lb_{12},\lb_{11},\lb_{22},0,0,0).\label{sextuple-decompositions}
		\end{align}
		More precisely, $\mathcal{G}$ is the Zappa–Sz\'ep product $(G_1\rtimes G_2)G_3$ of its subgroups, where
		\begin{align*}
			G_1&=\{(\lb_{12},\lb_{11},\lb_{22},j,k,l)\in \mathcal{G}:\lb_{12}=\lb_{11}=\lb_{22}=1\text{ and }j=0\}\cong \ZZ\times\ZZ,\\
			G_2&=\{(\lb_{12},\lb_{11},\lb_{22},j,k,l)\in \mathcal{G}:\lb_{12}=\lb_{11}=\lb_{22}=1\text{ and }k=l=0\}\cong \ZZ,\\
			G_3&=\{(\lb_{12},\lb_{11},\lb_{22},j,k,l)\in \mathcal{G}:j=k=l=0\}\cong(K^*)^3.
		\end{align*}
		
		Indeed, all the decompositions of \cref{sextuple-decompositions} are straightforward by \cref{product-of-sextuples}, and their uniqueness is obvious by \cref{sextuple-decompositions} itself. It is directly verified that $G_3$ is a subgroup of $\mathcal{G}$ isomorphic to $(K^*)^3$ and that $\mathcal{G}=HG_3$, where 
		\begin{align*}
			H=\{(\lb_{12},\lb_{11},\lb_{22},j,k,l)\in \mathcal{G}:\lb_{12}=\lb_{11}=\lb_{22}=1\}  
		\end{align*}
		is a subgroup of $\mathcal{G}$ with $G_3\cap H=\{(1,1,1,0,0,0)\}$. Now, by the second and third equalities of \cref{sextuple-decompositions} we have $H=G_1G_2=G_2G_1$ and $G_1\trianglelefteq H$. Since, obviously, $G_1\cap G_2=\{(1,1,1,0,0,0)\}$, we conclude that $H=G_1\rtimes G_2$, whence $\mathcal{G}=(G_1\rtimes G_2)G_3$. Finally, the isomorphisms $G_1\cong \ZZ\times\ZZ$ and $G_2\cong\ZZ$ are evident by \cref{product-of-sextuples}.
	\end{remark}

	\begin{corollary}\label{Hopf-aut-UT_q(2)}
		Assume that $q$ is not a root of unity and let $z=a_{11}a_{22}^{-1}$. Then the group of Hopf algebra automorphisms of $UT_q(2)$ is
		\begin{align*}
			\mathcal{H}=\{\vf\in\Aut(UT_q(2)) : \vf(a_{12})=\lb z^ja_{12},\ \vf(a_{ii})=z^ja_{ii},\ i\in\{1, 2\}, \lb\in K^*, j\in\ZZ \}\cong K^*\times \ZZ.
		\end{align*}
	\end{corollary}
	\begin{proof}
		Let $\vf\in\Aut(UT_q(2))$. By \cref{Aut(UT_q(2))-cong-G-rtimes-<rho>} there are two cases.
		
		\textit{Case 1:} $\vf\in \mathcal{G}$. Then $\vf(a_{12})=\lb_{12}a_{11}^ka_{22}^l a_{12}$, $\vf(a_{ii})=\lb_{ii}z^ja_{ii}$, where $i\in\{1, 2\}$, $\lb_{ij}\in K^*$ and $j, k, l\in\ZZ$. For all $i\in\{1,2\}$ we have
		\begin{align*}
			(\vf\ot\vf)(\Dl(a_{ii}))&=\vf(a_{ii})\ot\vf(a_{ii})=\lb_{ii}z^ja_{ii}\ot \lb_{ii}z^ja_{ii}=\lb_{ii}^2z^ja_{ii}\ot z^ja_{ii},\\
			\Dl(\vf(a_{ii}))&=\Dl(\lb_{ii}z^ja_{ii})=\lb_{ii}z^ja_{ii}\ot z^ja_{ii},
		\end{align*}
		so that $(\vf\ot\vf)(\Dl(a_{ii}))=\Dl(\vf(a_{ii}))\iff\lb_{ii}=1$. Hence, assume $\lb_{11}=\lb_{22}=1$. Then
		\begin{align}
			(\vf\ot\vf)(\Dl(a_{12}))&=\vf(a_{11})\ot\vf(a_{12})+\vf(a_{12})\ot\vf(a_{22})=z^ja_{11}\ot\lb_{12}a_{11}^ka_{22}^la_{12}+\lb_{12}a_{11}^ka_{22}^l a_{12}\ot z^ja_{22},\notag\\
			\Dl(\vf(a_{12}))&=\Dl(\lb_{12}a_{11}^ka_{22}^la_{12})=\lb_{12}(a_{11}^ka_{22}^l\ot a_{11}^ka_{22}^l)(a_{11}\ot a_{12}+a_{12}\ot a_{22}),\label{Dl(vf(a_12))}
		\end{align}
		so that $(\vf\ot\vf)(\Dl(a_{12}))=\Dl(\vf(a_{12}))\iff z^j=a_{11}^ka_{22}^l$, in which case $\vf(a_{12})=\lb_{12}z^ja_{12}$. Thus, $\vf$ respects $\Dl$ $\iff\vf\in \mathcal{H}$. Since $\ve\circ\vf=\ve$ for any $\vf\in\mathcal{H}$, we are done.
		
		\textit{Case 2:} $\vf\in \mathcal{G}\rho$. Then $\vf(a_{12})=\lb_{12}a_{11}^ka_{22}^l a_{12}$, $\vf(a_{ii})=\lb_{ii}z^ja_{\ol i\,\ol i}$, where $i\in\{1, 2\}$, $\ol i=3-i$, $\lb_{ij}\in K^*$ and $j, k, l\in\ZZ$. As in Case 1 one has $(\vf\ot\vf)(\Dl(a_{ii}))=\Dl(\vf(a_{ii}))\iff\lb_{ii}=1$ for all $i\in\{1,2\}$. Assuming $\lb_{11}=\lb_{22}=1$, we have
		\begin{align*}
			(\vf\ot\vf)(\Dl(a_{12}))&=\vf(a_{11})\ot\vf(a_{12})+\vf(a_{12})\ot\vf(a_{22})=z^ja_{22}\ot\lb_{12}a_{11}^ka_{22}^la_{12}+\lb_{12}a_{11}^ka_{22}^l a_{12}\ot z^ja_{11},
		\end{align*}
		while $\Dl(\vf(a_{12}))$ is given by \cref{Dl(vf(a_12))} as in Case 1. Hence $(\vf\ot\vf)(\Dl(a_{12}))=\Dl(\vf(a_{12}))\iff z^ja_{22}=a_{11}^{k+1}a_{22}^l$ and $z^ja_{11}=a_{11}^ka_{22}^{l+1}$, which is impossible. Thus, $\vf$ is not a bialgebra morphism.
		
		The isomorphism $\mathcal{H}\cong K^*\times \ZZ$ is obvious.
	\end{proof}
	
	\begin{corollary}
		Assume that $q\in K_0$ is not a root of unity. Then the group of Hopf $*$-algebra automorphisms of $UT_q(2)$ is
		\begin{align*}
			\mathcal{K}=\{\vf\in\Aut(UT_q(2)) : \vf(a_{12})=\lb a_{12},\ \vf(a_{ll})=a_{ll},\ l\in\{1, 2\}, \lb\in K_0^*\}\cong K_0^*.
		\end{align*}
	\end{corollary}
	\begin{proof}
		Observe from \cref{Hopf-*-str-on-UT_q(n)} that
		\begin{align*}
			a_{12}^*=-qa\m_{11}a\m_{22}a_{12},\quad a_{11}^*=a\m_{22}\quad\text{and}\quad a_{22}^*=a\m_{11}.
		\end{align*}
		Let $\vf\in\Aut(UT_q(2))$ be a Hopf $*$-algebra automorphism of $UT_q(2)$. By \cref{Hopf-aut-UT_q(2)} there are $\lb\in K^*$ and $j\in\ZZ$ such that $\vf(a_{12})=\lb z^ja_{12}$ and $\vf(a_{ll})=z^ja_{ll}$ for all $l\in\{1, 2\}$, where $z=a_{11}a_{22}^{-1}$. Since $z^*=(a_{11}a_{22}^{-1})^*=(a_{22}^*)\m a_{11}^*=a_{11}a\m_{22}=z$, we have $\vf(a_{11})^*=(z^ja_{11})^*=a_{11}^*z^j=a\m_{22}z^j$, while $\vf(a_{11}^*)=\vf(a\m_{22})=(z^ja_{22})\m=a\m_{22}z^{-j}$. It follows that $j=0$, so that $\vf(a_{ll})=a_{ll}$ for $l\in\{1, 2\}$. Now, $\vf(a_{12})^*=(\lb a_{12})^*=\ol\lb a_{12}^*=-\ol\lb qa\m_{11}a\m_{22}a_{12}$ and $\vf(a_{12}^*)=\vf(-qa\m_{11}a\m_{22}a_{12})=-\lb qa\m_{11}a\m_{22}a_{12}$, yielding $\ol\lb=\lb$, i.e. $\lb\in K_0$. Thus, $\vf\in\mathcal{K}$. Conversely, if $\vf\in\mathcal{K}$, then it is easy to see that $\vf$ is a Hopf $*$-algebra automorphism of $UT_q(2)$.
	\end{proof}
	\section*{Acknowledgements}
	The authors thank Uli Kr\"ahmer and Lucas Buzaglo for discussions and comments on a previous version of the manuscript, which helped improve it.
	The second and third authors were partially supported by CMUP -- Centro de Matem\'atica da Universidade do Porto, member of LASI, which is financed by national funds through FCT -- Funda\c c\~ao para a Ci\^encia e a Tecnologia, I.P., under the project with reference UID/00144/2025, doi: \url{https://doi.org/10.54499/UID/00144/2025}. 
	\bibliography{bibl}{}
	\bibliographystyle{acm}
	
\end{document}